\documentclass[9pt]{article}

\usepackage[fleqn]{amsmath}
\usepackage{amsthm,amsfonts,amssymb,amscd}
\DeclareMathSizes{10}{9}{7}{5}
\usepackage{subfigure}
\usepackage{graphicx,bm,color}
\usepackage{algorithm}
\usepackage{algpseudocode}
\usepackage{stmaryrd}
\usepackage{listings}
\usepackage{lineno}
\usepackage{mathrsfs}
\usepackage{lipsum}
\usepackage{url}            
\usepackage{booktabs}
\usepackage{authblk}
\usepackage[font=small,labelfont=bf]{caption}
\usepackage[justification=centering]{caption}
\usepackage{titlesec}
\titleformat{\section}{\large\bfseries}{\thesection}{1em}{}
\titleformat{\appendix}{\large\bfseries}{\thesection}{1em}{}
\numberwithin{equation}{section}

\usepackage[numbers]{natbib}
\definecolor{gray}{gray}{0.6}

\newtheorem{assumption}{Assumption}
\newtheorem{lemma}{Lemma}

\newcommand{\vertiii}[1]{{\left\vert\kern-0.25ex\left\vert\kern-0.25ex\left\vert #1
    \right\vert\kern-0.25ex\right\vert\kern-0.25ex\right\vert}}

\begin{document}

\title{Sequential Local Mesh Refinement Solver with Separate Temporal and Spatial Adaptivity for Non-linear Two-phase Flow Problems}

\author{Hanyu Li \quad Wing Tat Leung \quad Mary F. Wheeler}

\date{\today}
\maketitle

\begin{abstract}
Convergence failure and slow convergence rates are among the biggest challenges with solving the system of non-linear equations numerically. Although mitigated, such issues still linger when using strictly small time steps and unconditionally stable fully implicit schemes. The price that comes with restricting time steps to small scales is the enormous computational load, especially in large-scale models. To address this problem, we introduce a sequential local mesh refinement framework to optimize convergence rate and prevent convergence failure, while not restricting the whole system to small time steps, thus improving computational efficiency. We test the algorithm with the non-linear two-phase flow model. Starting by solving the problem on the coarsest space-time mesh, the algorithm refines the domain in time at water saturation front to prevent convergence failure. It then deploys fine spatial grid in regions with large saturation variation to assure solution accuracy.  After each refinement, the solution from the previous mesh is used to estimate the initial guess of unknowns on the current mesh for faster convergence. Numerical results are presented to confirm accuracy of our algorithm as compared to the uniformly fine time step and fine spatial discretization solution. We observe approximately 25 times speedup in the solution time by using our algorithm.
\end{abstract}
\begin{keyword}


Space-time domain decomposition, Mixed finite element method, Sequential local mesh refinement, Iterative solver, Non-linear system
\end{keyword}

\newcommand{\bs}[1]{\boldsymbol{#1}}
\section{Introduction}
\label{sec:int}

\par Complex multi-phase flow and reactive transport in subsurface porous media is mathematically modeled by a system of non-linear equations. Due to the significant non-linearity, solving such system with Newton's method usually suffers from convergence issues even when applying strictly small time steps and using unconditionally stable fully implicit schemes. This problem becomes much more severe in large-scale models. The enormous number of unknowns makes each Newton iteration computationally exhaustive. Therefore by reducing the size of the model using multiscale techniques and optimizing the convergence rate of Newton's method, we can achieve orders of magnitude greater computational efficiency.
\par Adaptive homogenization \cite{Amanbek:0619, Singh:1118} reduces the number of unknowns in the model by replacing fine mesh with coarse mesh in regions where non-linearity and variable (eg. saturation) variation is negligible. However, fine and coarse discretization in space requires different time scales for stable numerical solution. Forcing the coarse mesh to accommodate the fine mesh by taking fine time steps fails to reduce the number of unknowns in time. Space-time domain decomposition addresses this issue by allowing different time scales for different spatial grid. Several space-time domain decomposition approaches have been proposed in the past. \cite{Hughes:0288, Hulbert:1290} introduced space-time finite element method for elastodynamics with discontinuous Galerkin (DG) in time. The method has also been applied to other types of problems such as diffusion with different time discretization schemes \cite{Bause:1215, Bause:0617, Kocher:15, Kocher:1114}.
\par The aforementioned literatures applied space-time domain decomposition method to mostly mechanics problems. On the other hand, prior work regarding flow mainly focused on linear single phase flow and transport problems where flow is naturally decoupled from the advection-diffusion component transport \cite{Hoang:1213, Hoang:0717}. \cite{Singh:0818} first validates the space-time approach for non-linear coupled multiphase flow and transport problems on static grid. \cite{Singh:0818} enforces strong continuity of fluxes at non-matching space-time interfaces with enhanced velocity introduced in \cite{Thomas:0911, Wheeler:0102, Amanbek:18, Amanbek:0119}. It also constructs and solves a monolithic system to avoid computational overheads associated with iterative solution schemes introduced in \cite{Hoang:1213}, that require subdomain problems to be solved iteratively until weak continuity of fluxes is satisfied at interfaces. \cite{Singh:0918} upgrades the method by allowing adaptive mesh refinement, thus improving computational efficiency while maintaining accuracy as compared to the uniformly fine scale solution. The upgraded algorithm uses initial residual as a cheap error estimator to search for regions that need refinement. A number of other error estimators are mentioned in \cite{Pencheva:0213, Vohralik:1013}. As shown in Fig.\ref{fig:satres}, the normalized non-linear residual becomes the largest in regions with the highest non-linearity (water saturation front) and thus consumes most computational resources. Refining such regions in time ensures Newton convergence while refining in space maintains solution accuracy.
\begin{figure}[t]
  \center{\includegraphics[width=\textwidth]{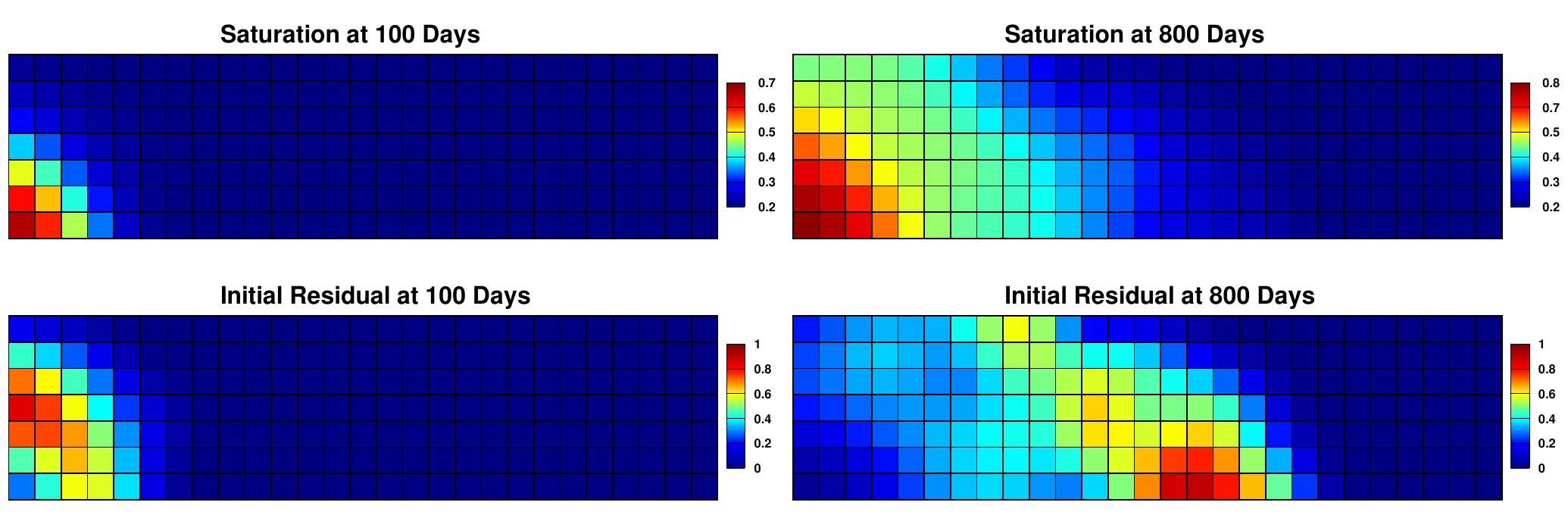}}
  \setlength{\abovecaptionskip}{-10pt}
  \caption{\bf{Saturation and normalized initial non-linear residual at 100 and 800 days}}
  \label{fig:satres}
\end{figure}
\par The adaptive local mesh refinement approach demonstrated by \cite{Singh:0918} allows only one level of refinement in both space and time, thus restricting the largest coarse time step allowed for stable numerical convergence. \cite{Li:0119} extended such approach by allowing more refinement levels, similar to the algorithm introduced in \cite{Chueh:1010}, which however only considers spatial adaptivity. When solving problems on each coarse space-time domain, regions with large non-linear residual and saturation variation are sequentially refined to the finest resolution to ensure solution convergence and accuracy. After each refinement, before solving the problem on the new mesh, the initial guess for the unknowns are populated by the solution on the previous mesh using linear projection. The initial guess provided in such manner is naturally closer to the true solution. Therefore, the non-linear solver convergence is not only guaranteed, but also accelerated. Although achieving 5 times speedup on solution time with iterative linear solver, \cite{Li:0119} relied on isotropic space-time refinement which produces a significant number of unnecessary elements. Regions with large saturation variation behind the front are forced to take redundant fine time steps. Preventing such over-refinement will further improve computational performance. Another problem associated with isotropic refinement scheme is that, the error indicator used to pinpoint refinement location combines both temporal and spatial saturation variations. Error indicator calculated in this fashion sometimes misleads the refinement process, especially in channelized permeability field, thus damaging the solution accuracy and hindering numerical convergence. In this work, we extend the method demonstrated in \cite{Li:0119} by separating temporal and spatial adaptivity to further improve computational performance and solution accuracy.
\par In this paper, we restrict ourselves to non-linear two-phase flow problems in subsurface porous media. We intend to approach more complicated non-linear models such as three-phase black oil model in the near future. The rest of the paper begins by describing the governing equations for two phase flow, the functional spaces for space-time domain decomposition, the enhanced velocity weak variational formulation and its fully discrete form in Section \ref{sec:math}. Then we will discuss the error estimator analysis used for the refinement process in Section \ref{sec:err}. Following the math ground work, we present the algorithm for the sequential local mesh refinement solver in Section \ref{sec:sol}, including the mechanism to search for refinement regions and separately adapt the mesh in temporal and spatial dimensions. Afterwards, we demonstrate results from two numerical experiments using the proposed algorithm in Section \ref{sec:num}.

\section{Two phase flow formulation}
\label{sec:math}

\par We consider the following well-known two-phase, slightly compressible flow in porous medium model, with oil and water phase mass conservation, constitutive equations, boundary and initial conditions.
\begin{equation}
  \frac{\partial (\phi\rho_{\alpha}s_{\alpha})}{\partial t}+\nabla\cdot\bs{u}_{\alpha}=q_{\alpha} \quad in\ \Omega\, \times\, J
  \label{eq:mb}
\end{equation}
\begin{equation}
  \bs{u}_{\alpha}=-K\rho_{\alpha}\frac{k_{r\alpha}}{\mu_{\alpha}}(\nabla p_{\alpha}-\rho_{\alpha}\bs{g}) \quad in\ \Omega\, \times\, J
  \label{eq:da}
\end{equation}
\begin{equation}
  \bs{u}_{\alpha}\cdot\bs{\nu}=0 \quad on\ \partial\Omega\, \times\, J
  \label{eq:bc}
\end{equation}
\begin{equation}
  \begin{cases}
  \begin{aligned}
    p_{\alpha}=p_{\alpha}^{0} \\
    s_{\alpha}=s_{\alpha}^{0}
  \end{aligned}
  \end{cases}
  at\ \Omega\, \times\, \{t=0\}
  \label{eq:ic}
\end{equation}
$\phi$ is porosity and $K$ is permeability tensor. $\rho_{\alpha}$, $s_{\alpha}$, $\bs{u}_{\alpha}$ and $q_{\alpha}$ are density, saturation, velocity and source/sink, respectively for each phase. The phases are slightly compressible and the phase densities are calculated by Eqn.\eqref{eq:comp},
\begin{equation}
  \rho_{\alpha}=\rho_{\alpha,ref}\cdot e^{c_{f,\alpha}(p_{\alpha}-p_{\alpha,ref})}
  \label{eq:comp}
\end{equation}
with $c_{f,\alpha}$ being the fluid compressibility and $\rho_{\alpha,ref}$ being the reference density at reference pressure $p_{\alpha,ref}$. In the constitutive equation \eqref{eq:da} given by Darcy's law, $k_{r\alpha}$, $\mu_{\alpha}$ and $p_{\alpha}$ are the relative permeability, viscosity and pressure for each phase. Relative permeability is a function of saturation. Pressure differs between wetting phase and non-wetting phase in the presence of capillary pressure, which is also a function of saturation.
\begin{equation}
  k_{r\alpha}=f(s_{\alpha})
  \label{eq:relperm}
\end{equation}
\begin{equation}
  p_{c}=g(s_{\alpha})=p_{nw}-p_{w}
  \label{eq:cap}
\end{equation}
The saturation of all phases obeys the constrain \eqref{eq:satcon}.
\begin{equation}
  \sum_{\alpha} s_{\alpha}=1
  \label{eq:satcon}
\end{equation}
The boundary and initial conditions are given by Eqn.\eqref{eq:bc} and \eqref{eq:ic}. $J=(0,T]$ is the time domain of interest while $\Omega$ is the spatial domain.
\par Now we will give a brief introduction of enhanced velocity formulation in space-time domain. Let $J=(0,T]$ be partitioned in to a number of coarse time intervals $\{t_{n}\}_{n=1}^{N}$ where $0=t_{1}<t_{2}<\cdots<t_{N}=T$. $J_{n}=(t_{n},t_{n+1}]$ is the nth partition of the time domain of interest. Consider $J_{n} \times \Omega$ as an union of some non-overlapping subdomains $\big\{I_{m} \times \Omega_{i}\big\}$, namely $J_{n} \times \Omega=\cup_{i,m}\big(I_{m} \times \Omega_{i}\big)$, where $I_{m}=(\tau_{m},\tau_{m+1}]$ is a sub-interval of $J_{n}=(t_{n},t_{n+1}]$ and $\Omega_{i}$ is a subdomain of $\Omega$. The interfaces of the subdomains are defined as $\Gamma_{i,j}^{m,s}=\partial \Big( I_{m}\times\Omega_{i}\Big)\cap \partial \Big( I_{s}\times\Omega_{j}\Big)$, $\Gamma=\cup_{m,s,i,j}\Gamma_{i,j}^{m,s}$ and $\Gamma_{i}^{m}= \Gamma\cap \partial \Big( I_{m}\times\Omega_{i}\Big)$. We use space-time enhanced velocity method similar as \cite{Wheeler:0102} to discretize the system. The functional spaces for mixed weak formulation are
\begin{equation*}
  \bs{V}=H(div;\Omega)=\bigg\{\bs{v}\in\big(L^2(\Omega)\big)^d:\nabla\cdot\bs{v}\in L^2(\Omega)\bigg\},\quad W=L^2(\Omega),
\end{equation*}
with finite dimensional subspaces as $\bs{V}_{h}$ and $W_{h}$. Let $\mathcal{T}_{h,i}^{n,m}$ be a rectangular partition of $I_{m} \times \Omega_{i}$, $E=T \times F$ be an space-time element in such partition and $\mathcal{T}_{h}^{n}=\cup_{i,m}\mathcal{T}_{h,i}^{n,m}$. Define velocity and pressure/saturation spaces as
\begin{equation*}
  \bs{V}_{h,i}^{n,m}=\Bigg\{\bs{v}\in L^2\Big(I_{m};H(div;\Omega_{i})\Big):\bs{v}(\cdot,\bs{x})\Big|_{F} \in \bs{V}_{h}(F)\ and\ \bs{v}(t,\cdot)\Big|_{T}=\displaystyle \sum_{a=1}^{l} \bs{v}_{a}t^{a}\ with\ \bs{v}_{a}\in \bs{V}_{h}(F), \forall\ E \in \mathcal{T}_{h,i}^{n,m}\Bigg\},
\end{equation*}
\begin{equation*}
  W_{h,i}^{n,m}=\Bigg\{w\in L^2\Big(I_{m};L^2(\Omega_{i})\Big):w(\cdot,\bs{x})\Big|_{F} \in W_{h}(F)\ and\ w(t,\cdot)\Big|_{T}=\displaystyle \sum_{a=1}^{l} w_{a}t^{a}\ with\ w_{a}\in W_{h}(F), \forall\ E \in \mathcal{T}_{h,i}^{n,m}\Bigg\}.
\end{equation*}
where
\begin{equation*}
  \bs{V}_h(F)= \bigg\{\bs{v}=(v_1,v_2)\text{ or }\bs{v}=(v_1,v_2,v_3):v_l=\alpha_l+\beta_lx_l;\;\alpha_l,\beta_l\in \mathbb{R},l=1,\cdots,d\bigg\}
\end{equation*}
\begin{equation*}
  W_h(F)= \bigg\{w\text{ is a constant in }F\bigg\}
\end{equation*}
Functions in $\bs{V}_{h,i}^{n,m}$ and $W_{h,i}^{n,m}$ along time dimension are represented by polynomials with degrees up to $l$. As described in \cite{Singh:0818}, following the discontinuous Galerkin (DG) discretization in time \cite{Arbogast:0215, Johnson:1987}, the $DG_{0}$ (polynomial of degree zero) scheme makes $\bs{v}(t,\cdot)\big|_{T}$ and $w(t,\cdot)\big|_{T}$ constant. Then we define the product spaces as $\bs{V}_{h}^{n}= \oplus_{i,m}\bs{V}_{h,i}^{n,m}$. We remark that $\bs{V}_{h}^{n}$ is not a subspace of $\bs{V}$. To obtain a finite element space containing basis functions with continuous normal flux, we need to modify the basis functions on the space-time interface $\Gamma_{i,j}^{m,s}$. Let $\mathcal{E}_{h,i,j}^{n,m,s}$ be the rectangular partition of $\Gamma_{i,j}^{m,s}$ obtained from the intersection of the traces of $\mathcal{T}_{h,i}^{n,m}$ and $\mathcal{T}_{h,j}^{n,s}$. For each $e\in \mathcal{E}_{h,i,j}^{n,m,s}$, we define a $RT_0$ basis function $\bs{v}_{e}$ with a normal component equal to one on $e$ , namely $\bs{v}_{e}\big|_e\cdot \bs{\nu}=1$. We then define the space $\bs{V}^{\Gamma}_h$ to be the  span of all these basis function, $\bs{v}_e$. Then the space-time mixed finite element velocity space $\bs{V}^{n,*}_h$ is
\begin{equation*}
  \bs{V}^{n,*}_h= \Big(\oplus_{i,m}\bs{V}_{h,i}^{n,m,0}\Big)\oplus \bs{V}^{\Gamma}_h
\end{equation*}
where $\bs{V}^{n,m,0}_{h,i}$ is the subspace of $\bs{V}_{h,i}^{n,m}$ with zero normal component on $\Gamma_{i}^{m}$. Similarly, the pressure/saturation space is $W_{h}^n=\oplus_{i,m}W_{h,i}^{n,m}$.
\par Now consider any function $f$ piecewise in time (eg. functions in $\bs{V}_{h}^{n,*}$ and $W_{h}^{n}$), define $f_{\tau}$ as the linear interpolation along time direction as
\begin{equation*}
  \displaystyle f_{\tau}(t,\cdot)\Big|_{(\tau_{m},\tau_{m+1}]\times\Omega_{i}}=\frac{t-\tau_{m}}{\tau_{m+1}-\tau_{m}}f(\tau_{m+1}^{-},\cdot)\Big|_{\Omega_{i}}+\frac{\tau_{m+1}-t}{\tau_{m+1}-\tau_{m}}f(\tau_{m}^{-},\cdot)\Big|_{\Omega_{i}},
\end{equation*}
and we have
\begin{equation*}
  \displaystyle \int_{\tau_{m}}^{\tau_{m+1}}\partial_{t}f_{\tau}(t,\cdot)=f(\tau_{m+1},\cdot)-f(\tau_{m},\cdot)
\end{equation*}
To simplify the notation, let $C_{\alpha,h}^{n}=\rho_{\alpha}\Big(p_{\alpha,h}^{n}\Big)s_{\alpha,h}^{n}$ be the phase mass concentration. Then space-time enhanced velocity method formulates Eqn.\eqref{eq:mb} and \eqref{eq:da} as: find $\bs{u}_{\alpha,h}^{n}\in \bs{V}_{h}^{n,*}$, $\tilde{\bs{u}}_{\alpha,h}^{n}\in \bs{V}_{h}^{n,*}$, $s_{\alpha,h}^{n}\in W_{h}^{n}$, $p_{\alpha,h}^{n}\in W_{h}^{n}$ such that
\begin{equation}
  \int_{J_{n}}\int_{\Omega}\partial_{t}\Big(\phi C_{\alpha,h,\tau}^{n}\Big)w+\int_{J_{n}}\int_{\Omega}\Big(\nabla\cdot\bs{u}_{up,\alpha,h}^{n}\Big)w=\int_{J_{n}}\int_{\Omega}q_{\alpha}w \quad \forall w \in W_{h}^{n}
  \label{eq:evmb}
\end{equation}
\begin{equation}
  \int_{J_{n}}\int_{\Omega}K^{-1}\tilde{\bs{u}}_{\alpha,h}^{n}\cdot\bs{v}=\int_{J_{n}}\int_{\Omega}p_{\alpha,h}^{n}\nabla\cdot\bs{v} \quad \forall \bs{v} \in \bs{V}_{h}^{n,*}
  \label{eq:evda}
\end{equation}
\begin{equation}
  \int_{J_{n}}\int_{\Omega}\bs{u}_{\alpha,h}^{n}\cdot\bs{v}=\int_{J_{n}}\int_{\Omega}\lambda_{\alpha}\tilde{\bs{u}}_{\alpha,h}^{n}\cdot\bs{v} \quad \forall \bs{v} \in \bs{V}_{h}^{n,*}
  \label{eq:aux}
\end{equation}
The mobility ratio in \eqref{eq:aux} is defined as
\begin{equation}
  \lambda_{\alpha}=\frac{k_{r\alpha}\rho_{\alpha}}{\mu_{\alpha}}
  \label{eq:mob}
\end{equation}
and $\bs{u}_{up,\alpha,h}^{n}$ is the upwind velocity calculated by
\begin{equation}
  \int_{J_{n}}\int_{\Omega}\bs{u}_{up,\alpha,h}^{n}\cdot\bs{v}=\int_{J_{n}}\int_{\Omega}\lambda_{\alpha}^{*}\tilde{\bs{u}}_{\alpha,h}^{n}\cdot\bs{v} \quad \forall \bs{v} \in \bs{V}_{h}^{n,*}
  \label{eq:upvel}
\end{equation}
The additional auxiliary phase fluxes $\tilde{\bs{u}}_{\alpha,h}^{n}$ is used to avoid inverting zero phase relative permeability \cite{Peszy:0306}. Calculation of the upwind mobility ratio $\lambda_{\alpha}^{*}$ is referred to Eqn.\eqref{eq:uw}. The variational form for specifically oil-water system and its fully discrete formulation is attached in \ref{sec:ful}. The discrete formulation provides us a non-linear system of equations for pressure and saturation. We approximate such system in linear form and use Newton iteration to approach the true solution. Depending on the level of non-linearity and the closeness between the initial guess and the true solution, Newton's method could take numerous iterations before achieving convergence. We will use the sequential local mesh refinement algorithm to accelerate the Newton convergence. Before introducing refinement algorithm, in the next section, we will first present analysis for error estimator used for searching refinement regions.

\section{A Posteriori Error Estimate}
\label{sec:err}

\par In this section, we discuss the error estimate analysis as an extension to the work presented in \cite{Vohralik:1013}. In contrast to the previous work, our approach to calculate a posteriori error estimate does not rely on computationally expensive local reconstruction of fine scale solution from coarse scale solution. In this section, let $E_{i}^{m}=(\tau_{m},\tau_{m+1}] \times \Omega_{i} \in \mathcal{T}_{h}^n$ be a space-time element, we define local error estimators $\eta_{t,r,\alpha,i}^{n,m}, \eta_{s,r,\alpha,i}^{n,m}, \eta_{t,f,\alpha,i}^{n,m}, \eta_{s,f,\alpha,i}^{n,m}$, $\eta_{t,p,\alpha,i}^{n,m}$, $\eta_{s,p,\alpha,i}^{n,m}$ as follow
\begin{equation}
  \eta_{t,r,\alpha,i}^{n,m}=\big|\tau_{m+1}-\tau_{m}\big|\Bigg(\int_{E_{i}^{m}}\Big|\partial_{t}\Big(\phi C_{\alpha,h,\tau}^{n}\Big)+\nabla\cdot \bs{u}_{up,\alpha,h,\tau}^{n}-q_{\alpha}\Big|^{2}\Bigg)^{\frac{1}{2}}
  \label{eq:tres}
\end{equation}
\begin{equation}
  \eta_{s,r,\alpha,i}^{n,m}=\big|\Omega_{i}\big|\Bigg(\int_{E_{i}^{m}}\Big|\partial_{t}\Big(\phi C_{\alpha,h,\tau}^{n}\Big)+\nabla\cdot \bs{u}_{up,\alpha,h,\tau}^{n}-q_{\alpha}\Big|^{2}\Bigg)^{\frac{1}{2}}
  \label{eq:sres}
\end{equation}
\begin{equation}
  \eta_{t,f,\alpha,i}^{n,m}=\Bigg(\int_{E_{i}^{m}}K^{-1}\Big|\bs{u}_{\alpha,h}^{n}-\bs{u}_{\alpha,h,\tau}^{n}\Big|^{2}\Bigg)^{\frac{1}{2}}
  \label{eq:tflux}
\end{equation}
\begin{equation}
  \eta_{s,f,\alpha,i}^{n,m}=\Bigg(\int_{E_{i}^{m}}K^{-1}\Big|\bs{u}_{up,\alpha,h}^{n}-\bs{u}_{\alpha,h}^{n}\Big|^{2}\Bigg)^{\frac{1}{2}}
  \label{eq:sflux}
\end{equation}
\begin{equation}
  \eta_{t,p,\alpha,i}^{n,m}=\Bigg(\int_{E_{i}^{m}}K^{-1}\Big|\tilde{\bs{u}}_{\alpha,h}^{n}-\tilde{\bs{u}}_{\alpha,h,\tau}^{n}\Big|^{2}\Bigg)^{\frac{1}{2}}
  \label{eq:tnoncon}
\end{equation}
\begin{equation}
  \eta_{s,p,\alpha,i}^{n,m}=\Bigg(\big|\Omega_{i}\big|^{2}\int_{E_{i}^{m}}\Big|\nabla\times\Big(K^{-1}\tilde{\bs{u}}_{\alpha,h}^{n}\Big)\Big|^{2}+\sum_{e\subset\overline{E_{i}^{m}}}|e|\int_{e}\Big[\Big(K^{-1}\tilde{\bs{u}}_{\alpha,h}^{n}\Big)\times \bs{n}_{e}\Big]^{2}\Bigg)^{\frac{1}{2}}
  \label{eq:snoncon}
\end{equation}
Eqn.\eqref{eq:tres}-\eqref{eq:sres} are residual estimators, Eqn.\eqref{eq:tflux}-\eqref{eq:sflux} are flux estimators and Eqn.\eqref{eq:tnoncon}-\eqref{eq:snoncon} are nonconformity estimators. Eqn.\eqref{eq:tres} through \eqref{eq:snoncon} will provide upper bound for the error measure we are about to introduce. It is common to use energy norm as an error measurement for linear problems. However, it is much more complicated for nonlinear problems. Instead we use the dual norm of the residual, which is also widely applied, as our error measure. We denote
\begin{equation*}
  X^{n}=L^{2}\bigg(J_n;H^{1}(\Omega)\bigg) \bigcap H^{1}\bigg(J_n;L^{2}(\Omega)\bigg),
\end{equation*}
and for any $\psi \in X^{n}$
\begin{equation*}
  \big\|\psi\big\|_{X^{n}}=\int_{J_n} \sum_{\Omega_{i} \in \Omega}\bigg(\big\|\psi\big\|_{L^{2}(\Omega_{i})}^{2}+\big\|K^{\frac{1}{2}}\nabla\psi\big\|_{L^{2}(\Omega_{i})}^{2}+\big\|\psi_{t}\big\|_{L^{2}(\Omega_{i})}^{2}\bigg).
\end{equation*}
Let $s_{\alpha}^{n}$, $p_{\alpha}^{n}$, $\bs{u}_{\alpha}^{n}$ and $s_{\alpha,h}^{n}$, $p_{\alpha,h}^{n}$, $\bs{u}_{\alpha,h}^{n}$ be the exact and numerical saturation, pressure and velocities solutions. The error measure $\vertiii{\cdot}$ is defined as
\begin{equation*}
  \vertiii{(s_{\alpha}^{n}-s_{\alpha,h}^{n},p_{\alpha}^{n}-p_{\alpha,h}^{n},\tilde{\bs{u}}_{\alpha}^{n}-\tilde{\bs{u}}_{\alpha,h}^{n})} \mathrel{\mathop:}= N^n_{\alpha}+N^n_{\alpha,p}
\end{equation*}
where
\begin{equation*}
  N_{\alpha}^{n}=\sup_{\psi \in X^{n}, \|\psi\|=1} \Bigg\{\int_{J_n}\bigg(\int_{\Omega}\partial_{t}\Big(\phi C_{\alpha,h}^{n}-\phi C_{\alpha,h,\tau}^{n}\Big)\psi-\int_{\Omega}\Big(\bs{u}_{\alpha,h}^{n}-\bs{u}_{\alpha,h,\tau}^{n}\Big)\cdot\nabla\psi\bigg)\Bigg\}.
\end{equation*}
and
\begin{equation*}
  N_{\alpha,p}^{n}=\inf_{\psi\in X^{n}}\Bigg\{\int_{J_n}\int_{\Omega}K^{-1}\Big(\tilde{\bs{u}}_{\alpha,h,\tau}^{n}+K\nabla\psi\Big)^{2}\Bigg\}^{\frac{1}{2}}
\end{equation*}
$N^n_{\alpha}$ represents the dual norm of the residual and $N^n_{\alpha,p}$ measures the non-nonconformity of the numerical solutions.
\begin{assumption}
\label{ass:assumption}
  We assume there exist a subspace $M_{h}^{n}\subset L^{2}\Big(J_n;H(curl,\Omega)\Big)$ such that
  \begin{equation*}
    \nabla\times M_{h}^{n}\subset \bs{V}^{n,*}_h\cap L^{2}\Big(J_n;H(div,\Omega)\Big).
  \end{equation*}
  Moreover, for all $v \in H(curl,\Omega)$, there exist $\Pi_{M^n_h}(v)\in M_{h}^{n}$ such that
  \begin{equation*}
    \begin{aligned}
      \big\|v-\Pi_{M^n_h}(v)\big\|_{L^{2}(E_{i}^{m})} & \leq C\ \big|\Omega_i\big|\ \big\|\nabla v\big\|_{L^{2}(E_{i}^{m})} \quad \forall E_i^m \in \mathcal{T}_h^n\\
      \big\|v-\Pi_{M^n_h}(v)\big\|_{L^{2}(e)}^{2} & \leq C\ |e|\ \|v\|_{H^{1/2}(e)}^{2} \quad \forall e \in E_{i}^m
    \end{aligned}
  \end{equation*}
  where $e$ is an edge of $E_i^m$ and $\Pi$ is the $L^2$ projection operator.
\end{assumption}
\par Next, we will present a posteriori error estimate for this error measure. In the following lemma, we estimate the dual norm of the residual of the mass balance equations.
\begin{lemma}
  Let $\eta^{n,m}_{t,r,\alpha,i},\eta^{n,m}_{s,r,\alpha,i}, \eta^{n,m}_{t,f,\alpha,i}$ and $\eta^{n,m}_{s,f,\alpha,i}$ be the error indicators defined in Eqn.\eqref{eq:tres}-\eqref{eq:sflux}. There exist constants $C, C_{poin}>0$ such that
  \begin{equation*}
    N_{\alpha}^{n} \leq \Bigg(\sum_{E_{i}^{m} \in \mathcal{T}_{h}^{n}}\Big(\eta_{t,f,\alpha,i}^{n,m}\Big)^{2}\Bigg)^{\frac{1}{2}}+\Bigg(\sum_{E_{i}^{m} \in \mathcal{T}_{h}^{n}}\Big(\eta_{s,f,\alpha,i}^{n,m}\Big)^{2}\Bigg)^{\frac{1}{2}}+C C_{poin}\Bigg\{\Bigg(\sum_{E_{i}^{m} \in \mathcal{T}_{h}^{n}}\Big(\eta_{t,r,\alpha,i}^{n,m}\Big)^{2}\Bigg)^{\frac{1}{2}}+\Bigg(\sum_{E_{i}^{m} \in \mathcal{T}_{h}^{n}}\Big(\eta_{s,r,\alpha,i}^{n,m}\Big)^{2}\Bigg)^{\frac{1}{2}}\Bigg\}
  \end{equation*}
\end{lemma}
\begin{proof}
  Since $s_{\alpha}^n$, $p_{\alpha}^n$ and $\bs{u}_{\alpha}^n$ are exact saturation, pressure and velocity, using the phase mass concentration formulation introduced in Section \ref{sec:math} to simplify the notation, for each $\psi \in X^{n}$ we have
  \begin{equation}
    \begin{split}
      &\int_{J_n}\bigg(\int_{\Omega}\partial_{t}\Big(\phi C_{\alpha}^{n}-\phi C_{\alpha,h,\tau}^{n}\Big)\psi-\int_{\Omega}\Big(\bs{u}_{\alpha}^{n}-\bs{u}_{\alpha,h,\tau}^{n}\Big)\cdot\nabla\psi\bigg) \\
      &=\int_{J_n}\bigg(\int_{\Omega}q_{\alpha}\psi-\partial_{t}\Big(\phi C_{\alpha,h,\tau}^{n}\Big)\psi+\bs{u}_{\alpha,h,\tau}^{n}\cdot\nabla\psi\bigg)
    \end{split}
    \label{eq:lem1_dualn}
  \end{equation}
  We split the term $\int_{J_n}\int_{\Omega}\bs{u}_{\alpha,h,\tau}^{n}\cdot\nabla\psi$ into two parts such that
  \begin{equation}
    \int_{J_n}\int_{\Omega}\bs{u}_{\alpha,h,\tau}^{n}\cdot\nabla\psi=\int_{J_n}\int_{\Omega}\bs{u}_{\alpha,h}^{n}\cdot\nabla\psi+\int_{J_n}\sum_{\Omega_{i} \in \Omega}\int_{\Omega_{i}}\Big(\bs{u}_{\alpha,h,\tau}^{n}-\bs{u}_{\alpha,h}^{n}\Big)\cdot\nabla\psi
    \label{eq:lem1_fluxt}
  \end{equation}
  Next, we split $\int_{J_n}\int_{\Omega}\bs{u}_{\alpha,h}^{n}\cdot\nabla\psi$ into two parts as
  \begin{equation}
    \int_{J_n}\int_{\Omega}\bs{u}_{\alpha,h}^{n}\cdot\nabla\psi=-\int_{J_n}\int_{\Omega}\Big(\nabla\cdot\bs{u}_{up,\alpha,h}^{n}\Big)\psi+\int_{J_n}\int_{\Omega}\Big(\bs{u}_{\alpha,h}^{n}-\bs{u}_{up,\alpha,h}^{n}\Big)\cdot\nabla\psi
    \label{eq:lem1_fluxup}
  \end{equation}
  Therefore, by Eqn.\eqref{eq:lem1_dualn}, \eqref{eq:lem1_fluxt} and \eqref{eq:lem1_fluxup}, the dual norm can be separated into three terms such that
  \begin{equation}
    \int_{J_n}\bigg(\int_{\Omega}\partial_{t}\Big(\phi C_{\alpha,h}^{n}-\phi C_{\alpha,h,\tau}^{n}\Big)\psi-\int_{\Omega}\Big(\bs{u}_{\alpha,h}^{n}-\bs{u}_{\alpha,h,\tau}^{n}\Big)\cdot\nabla\psi\bigg)=I_{1}+I_{2}+I_{3}
    \label{eq:lem1_ncomb}
  \end{equation}
  where
  \begin{equation}
    I_{1}=\sum_{E_{i}^{m} \in \mathcal{T}_{h}^{n}}\int_{E_{i}^{m}}q_{\alpha}\psi-\partial_{t}\Big(\phi C_{\alpha,h,\tau}^{n}\Big)\psi-\nabla\cdot\bs{u}_{up,\alpha,h}^{n}\psi
    \label{eq:lem1_i1}
  \end{equation}
  \begin{equation}
    I_{2}=\sum_{E_{i}^{m} \in \mathcal{T}_{h}^{n}}\int_{E_{i}^{m}}\Big(\bs{u}_{\alpha,h}^{n}-\bs{u}_{up,\alpha,h}^{n}\Big)\cdot\nabla\psi
    \label{eq:lem1_i2}
  \end{equation}
  \begin{equation}
    I_{3}=\sum_{E_{i}^{m} \in \mathcal{T}_{h}^{n}}\int_{E_{i}^{m}}\Big(\bs{u}_{\alpha,h,\tau}^{n}-\bs{u}_{\alpha,h}^{n}\Big)\cdot\nabla\psi
    \label{eq:lem1_i3}
  \end{equation}
  Since $s_{\alpha,h}^{n}$, $p_{\alpha,h}^{n}$, $\bs{u}_{\alpha,h}^{n}$ are the numerical solutions of Eqn.\eqref{eq:evmb} and \eqref{eq:evda} we have the following
  \begin{equation}
    \sum_{E_{i}^{m} \in \mathcal{T}_{h}^{n}}\int_{E_{i}^{m}}\bigg(q_{\alpha}-\partial_{t}\Big(\phi C_{\alpha,h,\tau}^{n}\Big)-\nabla\cdot\bs{u}_{up,\alpha,h}^{n}\bigg)\ w=0\quad \forall w\in W_{h}^{n}
    \label{eq:lem1_nzero}
  \end{equation}
  We take $w=\Pi_{W_{h}^{n}}\psi$ and by Poincaré inequality obtain the following bound for $I_{1}$
  \begin{equation}
    \begin{split}
      I_{1}&=\sum_{E_{i}^{m} \in \mathcal{T}_{h}^{n}}\int_{E_{i}^{m}}\bigg(q_{\alpha}-\partial_{t}\Big(\phi C_{\alpha,h,\tau}^{n}\Big)-\nabla\cdot\bs{u}_{up,\alpha,h}^{n}\bigg)\bigg(\psi-\Pi_{W_{h}^{n}}\psi\bigg) \\
      &\leq C_{poin}\sum_{E_{i}^{m} \in \mathcal{T}_{h}^{n}}\Big(\eta_{t,r,\alpha,i}^{n,m}+\eta_{s,r,\alpha,i}^{n,m}\Big)\Big(\|\nabla\psi\|_{L^{2}(E_{i}^{m})}+\|\psi_{t}\|_{L^{2}(E_{i}^{m})}\Big) \\
      &\leq CC_{poin}\Bigg\{\Bigg(\sum_{E_{i}^{m} \in \mathcal{T}_{h}^{n}}\Big(\eta_{s,f,\alpha,i}^{n,m}\Big)^{2}\Bigg)^{\frac{1}{2}}+\Bigg(\sum_{E_{i}^{m} \in \mathcal{T}_{h}^{n}}\Big(\eta_{t,r,\alpha,i}^{n,m}\Big)^{2}\Bigg)^{\frac{1}{2}}\Bigg\}\|\psi\|_{X^{n}}
    \end{split}
    \label{eq:lem1_i1bound}
  \end{equation}
  Next for $I_{2}$, by Cauchy\textendash Schwarz inequality we have
  \begin{equation}
    I_{2}=\sum_{E_{i}^{m} \in \mathcal{T}_{h}^{n}}\int_{E_{i}^{m}}\Big(\bs{u}_{\alpha,h}^{n}-\bs{u}_{up,\alpha,h}^{n}\Big)\cdot\nabla\psi \leq \sum_{E_{i}^{m} \in \mathcal{T}_{h}^{n}}\eta_{s,f,\alpha,i}^{n,m}\|K^{\frac{1}{2}}\nabla\psi\|_{L^{2}(E_{i}^{m})} \leq \Bigg(\sum_{E_{i}^{m} \in \mathcal{T}_{h}^{n}}\Big(\eta_{s,f,\alpha,i}^{n,m}\Big)^{2}\Bigg)^{\frac{1}{2}}\|\psi\|_{X^{n}}
    \label{eq:lem1_i2bound}
  \end{equation}
  Similarly, $I_{3}$ goes as
  \begin{equation}
      I_{3}=\sum_{E_{i}^{m} \in \mathcal{T}_{h}^{n}}\int_{E_{i}^{m}}\Big(\bs{u}_{\alpha,h,\tau}^{n}-\bs{u}_{\alpha,h}^{n}\Big)\cdot\nabla\psi \leq \Bigg(\sum_{E_{i}^{m} \in \mathcal{T}_{h}^{n}}\Big(\eta_{t,f,\alpha,i}^{n,m}\Big)^{2}\Bigg)^{\frac{1}{2}}\|\psi\|_{X^{n}}
    \label{eq:lem1_i3bound}
  \end{equation}
  By the definition of $N_{\alpha}$, the inequality for the dual norm of the residual is proved.
\end{proof}
\noindent In the following lemma, we will provide an upper bound estimate for the non-nonconformity error measure.
\begin{lemma}
Assuming $K\in C^1(\Omega)$, there exist constant C such that
  \begin{equation*}
    N_{\alpha,p}^{n}\leq C\Bigg\{\Bigg(\sum_{E_{i}^{m}\in\mathcal{T}_{h}^{n}}\Big(\eta_{t,p,\alpha,i}^{n,m}\Big)^{2}\Bigg)^{\frac{1}{2}}+\Bigg(\sum_{E_{i}^{m}\in\mathcal{T}_{h}^{n}}\Big(\eta_{s,p,\alpha,i}^{n,m}\Big)^{2}\Bigg)^{\frac{1}{2}}\Bigg\}
  \end{equation*}
\end{lemma}
\begin{proof}
  First, using Helmholtz decomposition, we have
  \begin{equation}
    K^{-1}\tilde{\bs{u}}_{\alpha,h}^{n}=\nabla\phi_{0}+\nabla\times\phi_{1}
    \label{eq:lem2_Helmholtz}
  \end{equation}
  with $\phi_0\in H^1(\Omega)$, $\phi_1\in H(curl,\Omega)$ for $d=3$ and $\phi_1\in H^1(\Omega)$ for $d=2$. Since $\psi \in X^n$, $\nabla\psi \in L^2(J_n \times \Omega)$. Meanwhile $\nabla\phi_0 \in L^2(\Omega)$ so $(\nabla\phi_0)_{\tau} \in L^2(J_n \times \Omega)$. Thus we have
  \begin{equation}
    \begin{aligned}
      \inf_{\psi\in X^{n}}\int_{J_n}\int_{\Omega}K^{-1}\Big(\tilde{\bs{u}}_{\alpha,h}^{n}+K\nabla\psi\Big)^{2} & \leq 2\int_{J_n}\int_{\Omega}K^{-1}\big|\tilde{\bs{u}}_{\alpha,h}^{n}-K\nabla\phi_{0}\big|^{2}+2\int_{J_n}\int_{\Omega}K\big|\nabla\phi_{0}-\big(\nabla\phi_{0}\big)_{\tau}\big|^{2}\\
      &=2\ \Bigg(\int_{J_n}\int_{\Omega}K\big|\nabla\times\phi_{1}\big|^{2}\Bigg)+2\big\|\tilde{\bs{u}}_{\alpha,h}^{n}-\tilde{\bs{u}}_{\alpha,h,\tau}^{n}\big\|_{L^{2}}^{2}
    \end{aligned}
    \label{eq:lem2_nonconfn}
  \end{equation}
  By Eqn.\eqref{eq:lem2_Helmholtz} and $\nabla\phi_0 \perp \nabla\times\phi_1$, we have
  \begin{equation}
    \begin{aligned}
      \int_{J_n}\int_{\Omega}\big|\nabla\times\phi_{1}\big|^{2}&=\int_{J_n}\int_{\Omega}\Big(K^{-1}\tilde{\bs{u}}_{\alpha,h}^{n}-\nabla\phi_{0}\Big)\cdot\Big(\nabla\times\phi_{1}\Big)\\
      &=\int_{J_n}\int_{\Omega}\Big(K^{-1}\tilde{\bs{u}}_{\alpha,h}^{n}\Big)\cdot\Big(\nabla\times\phi_{1}\Big)
    \end{aligned}
    \label{eq:lem2_curlphi1}
  \end{equation}
  Using Assumption \ref{ass:assumption}, there exist a $\Pi_{M^n_h}\phi_1\in M^n_h$ such that for all $E_{i}^{m}=(\tau_m,\tau_{m+1}]\times \Omega_i \in\mathcal{T}_{h}^{n}$,
  \begin{equation*}
    \begin{aligned}
      \big\|\phi_{1}-\Pi_{M^n_h}\phi_{1}\big\|^2_{L^{2}(E_{i}^{m})} & \leq C\big|\Omega_i\big|^2\big\|\nabla\phi_{1}\big\|_{L^{2}(E_{i}^{m})}^{2}\\
      \sum_{e\subset E_{i}^{m}}\big\|\phi_{1}-\Pi_{M^n_h}\phi_{1}\big\|_{L^{2}(e)}^{2} & \leq C\sum_{e\subset E_{i}^{m}}|e|\ \big\|\phi_{1}-\Pi_{M^n_h}\phi_{1}\big\|_{H^{1/2}(e)}^{2}\leq\sum_{e\subset E_{i}^{m}}C|e|\ \big\|\nabla\phi_{1}\big\|_{L^{2}(E_{i}^{m})}^{2}.
    \end{aligned}
  \end{equation*}
  Then $\nabla\times\Pi_{M^n_h}\phi_{1}\in V_{h}^{n,*}$, by Eqn.\eqref{eq:evda} we have
  \begin{equation*}
    \int_{J_n}\int_{\Omega}K^{-1}\tilde{\bs{u}}_{\alpha,h}^{n}\cdot\Big(\nabla\times\Pi_{M^n_h}\phi_{1}\Big)=\int_{J_n}\int_{\Omega}p_{\alpha,h}^n\nabla\cdot\Big(\nabla\times\Pi_{M^n_h}\phi_{1}\Big)=0
  \end{equation*}
  Therefore, we have
  \begin{equation}
    \begin{aligned}
      & \int_{J_n}\int_{\Omega}\Big(K^{-1}\tilde{\bs{u}}_{\alpha,h}^{n}\Big)\cdot\Big(\nabla\times\phi_{1}\Big)\\
      = & \int_{J_n}\int_{\Omega}\Big(K^{-1}\tilde{\bs{u}}_{\alpha,h}^{n}\Big)\cdot\Big[\nabla\times\Big(\phi_{1}-\Pi_{M^n_h}\phi_{1}\Big)\Big]\\
      \leq & \int_{J_n}\int_{\Omega}\nabla\times\Big(K^{-1}\tilde{\bs{u}}_{\alpha,h}^{n}\Big)\cdot\Big(\phi_{1}-\Pi_{M^n_h}\phi_{1}\Big)+\sum_{e\in\Gamma}\int_{J_n}\int_{e}\Big[\Big(K^{-1}\tilde{\bs{u}}_{\alpha,h}^{n}\Big)\times \bs{n}\Big]\cdot\Big(\phi_{1}-\Pi_{M^n_h}\phi_{1}\Big)\\
      \leq & \sum_{E_{i}^{m} \in \mathcal{T}_{h}^n}\Big\|\nabla\times\Big(K^{-1}\tilde{\bs{u}}_{\alpha,h}^{n}\Big)\Big\|_{L^{2}(E_{i}^{m})}\Big\|\phi_{1}-\Pi_{M^n_h}\phi_{1}\Big\|_{L^{2}(E_{i}^{m})}+\sum_{e \in \Gamma}\Big\|\Big(K^{-1}\tilde{\bs{u}}_{\alpha,h}^{n}\Big)\times \bs{n}\Big\|_{L^{2}(e)}\Big\|\phi_{1}-\Pi_{M^n_h}\phi_{1}\Big\|_{L^{2}(e)}
    \end{aligned}
    \label{eq:lem2_cs}
  \end{equation}
  Apply Eqn.\eqref{eq:lem2_cs} to Eqn.\eqref{eq:lem2_curlphi1} we obtain
  \begin{equation}
    \begin{aligned}
      \int_{J_n}\int_{\Omega}\big|\nabla\times\phi_{1}\big|^{2} \leq & C\ \Bigg(\sum_{E_{i}^m\in\mathcal{T}_{h}^n}\big|\Omega_i\big|^{2}\Big\|\nabla\times\Big(K^{-1}\tilde{\bs{u}}_{\alpha,h,\tau}^{n}\Big)\Big\|_{L^{2}(E_{i}^m)}^{2}+\sum_{e \in \Gamma}|e|\ \Big\|\Big(K^{-1}\tilde{\bs{u}}_{\alpha,h,\tau}^{n}\Big)\times \bs{n}\Big\|_{L^{2}(e)}^{2}\Bigg)^{\frac{1}{2}}\big\|\nabla\phi_{1}\big\|_{L^{2}(E_{i}^m)}^{2}\\
      \leq & C\ \Bigg(\sum_{E_{i}^m \in\mathcal{T}_{h}^n}\Big(\eta_{s,p,\alpha,i}^{n,m}\Big)^{2}\Bigg)^{\frac{1}{2}}\big\|\nabla\times\phi_{1}\big\|_{L^{2}(E_{i}^m)}
    \end{aligned}
    \label{eq:lem2_i1bound}
  \end{equation}
  Then the nonconformity error measure is bounded by
  \begin{equation}
    \begin{aligned}
      N_{\alpha,p}^{n} & \leq\inf_{\psi\in X^{n}}\Bigg\{\int_{J_n}\int_{\Omega}K^{-1}\Big(\tilde{\bs{u}}_{\alpha,h}^{n}+K\nabla\psi\Big)^{2}\Bigg\}^{\frac{1}{2}}+\Bigg(\int_{J_n}\int_{\Omega}K^{-1}\Big(\tilde{\bs{u}}_{\alpha,h}^{n}-\tilde{\bs{u}}_{\alpha,h,\tau}^{n}\Big)^{2}\Bigg)^{\frac{1}{2}}\\
      & \leq\int_{J_n}\int_{\Omega}\Big(K\big|\nabla\times\phi_{1}\big|^{2}\Big)+2\ \Bigg(\sum_{E_{i}^m\in\mathcal{T}_{h}^{n}}\Big(\eta_{t,p,\alpha,i}^{n,m}\Big)^{2}\Bigg)^{\frac{1}{2}}\\
      & \leq C\ \Bigg(\sum_{E_{i}^m\in\mathcal{T}_{h}^n}\Big(\eta_{s,p,\alpha,i}^{n,m}\Big)^{2}\Bigg)^{\frac{1}{2}}+2\ \Bigg(\sum_{E_{i}^m\in\mathcal{T}_{h}^{n}}\Big(\eta_{t,p,\alpha,i}^{n,m}\Big)^{2}\Bigg)^{\frac{1}{2}}
    \end{aligned}
    \label{eq:lem2_ncbound}
  \end{equation}
  The inequality of the nonconformity error measure is proved.
\end{proof}
The error estimator introduced in this section is used to search for refinement regions. In the next section, we will introduce our sequential local mesh refinement algorithm to reduce the size of the system and minimize the number of iterations required for convergence, while maintaining accuracy as compared to uniformly fine scale solution.

\section{Solution algorithm}
\label{sec:sol}

\par In this section we present the sequential local mesh refinement solver algorithm. The procedure starts by solving the problem at its coarsest resolution in space-time domain and then sequentially refines certain regions to its finest resolution. The coarsest time step is chosen such that the numerical convergence is guaranteed on the coarsest spatial grid. During the sequential refinement process, the solver first keeps the spatial mesh static at its coarsest level and searches for regions to refine in time. Once the last level of temporal refinement is implemented, the temporal discretization is finalized and the solver refines the mesh in space until reaching the finest resolution. Then the spatial grid is restored to the coarsest resolution, the solver marches forward in time with the coarsest time step and the whole process reiterates. The complete algorithm is illustrated in Fig.\ref{fig:algo}. We always start from the coarsest mesh and refines into deeper levels due to the tree data structure inherited from \cite{Li:0119, Ganis:0219, Finkel:0374}. The tree structure is represented by a group of pointers linked to each other. Allowing both refinement and agglomeration requires inserting and removing pointers in the middle of the tree and then re-associating hanging pointers. The toll caused by such complex operation will counteract the computational efficiency improvement. By constraining the operation to solely refinement, we are only required to evolve the tree by adding new levels on the bottom, which has a much smaller operation count.
\begin{figure}[t]
  \center{\includegraphics[scale=.75]{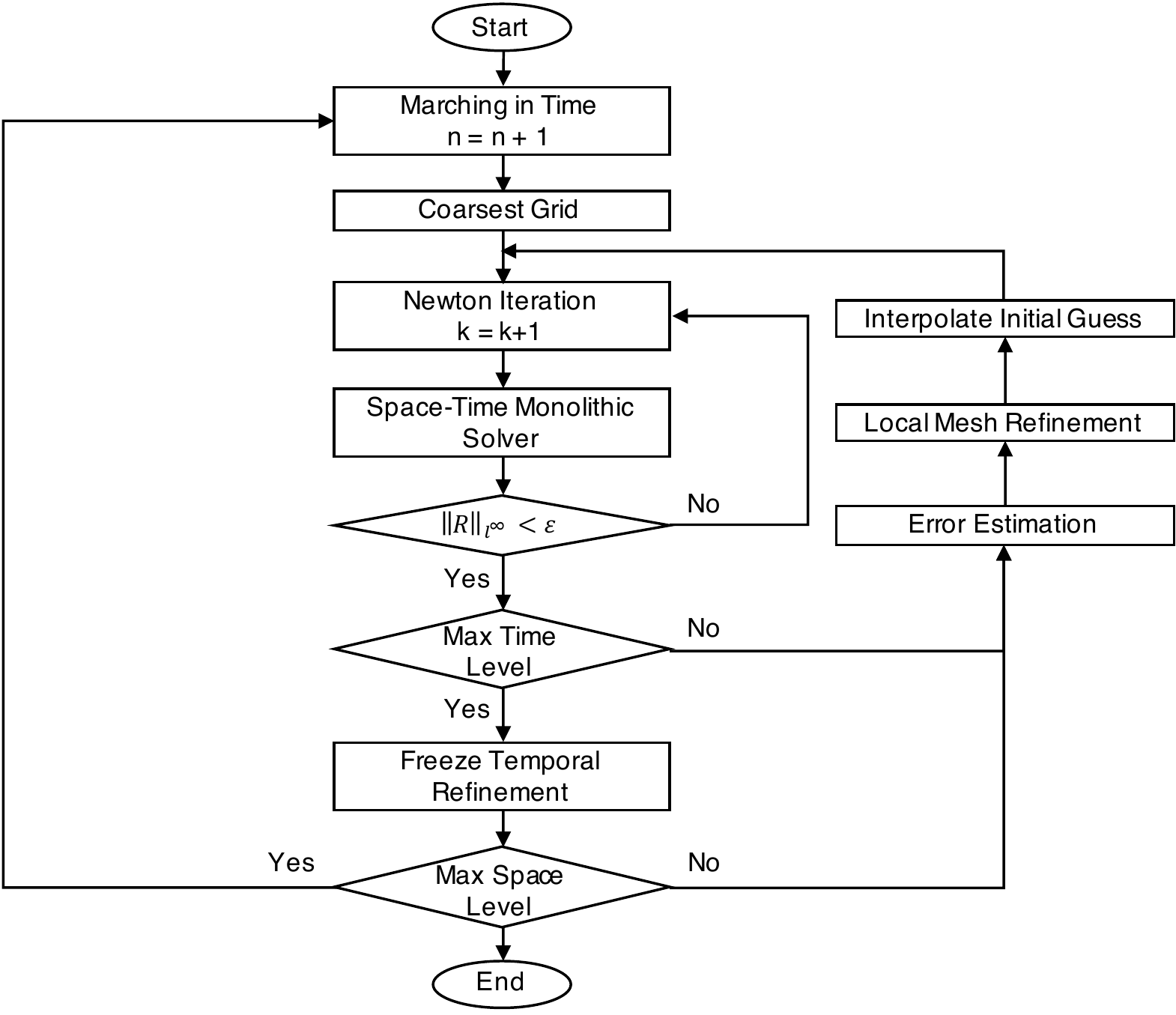}}
  \setlength{\abovecaptionskip}{5pt}
  \caption{\bf{Solution algorithm for sequential local mesh refinement solver with separate temporal and spatial adaptivity}}
  \label{fig:algo}
\end{figure}
\par Fig.\ref{fig:grid} demonstrates a sample semi-structured grid generated with the algorithm stated above. Here the $z$ axis represents time in a 2-D spatial problem. Separating temporal and spatial refinement makes the mesh construction more flexible. As observed from the plot, subdomains can have temporal refinement, spatial refinement or both. This flexibility reduces the total number of elements by two to three times as compared to the isotropic refinement scheme implemented in \cite{Li:0119}, thus improving computational performance. Note that the solver always refine spatially to the finest resolution for cells with well contained, for accurate estimate of production rate and bottom-hole flowing pressure. Adding temporal refinements for these cells depends on whether the saturation front is sweeping through the well or not.
\begin{figure}[t]
  \center{\includegraphics[width=\textwidth]{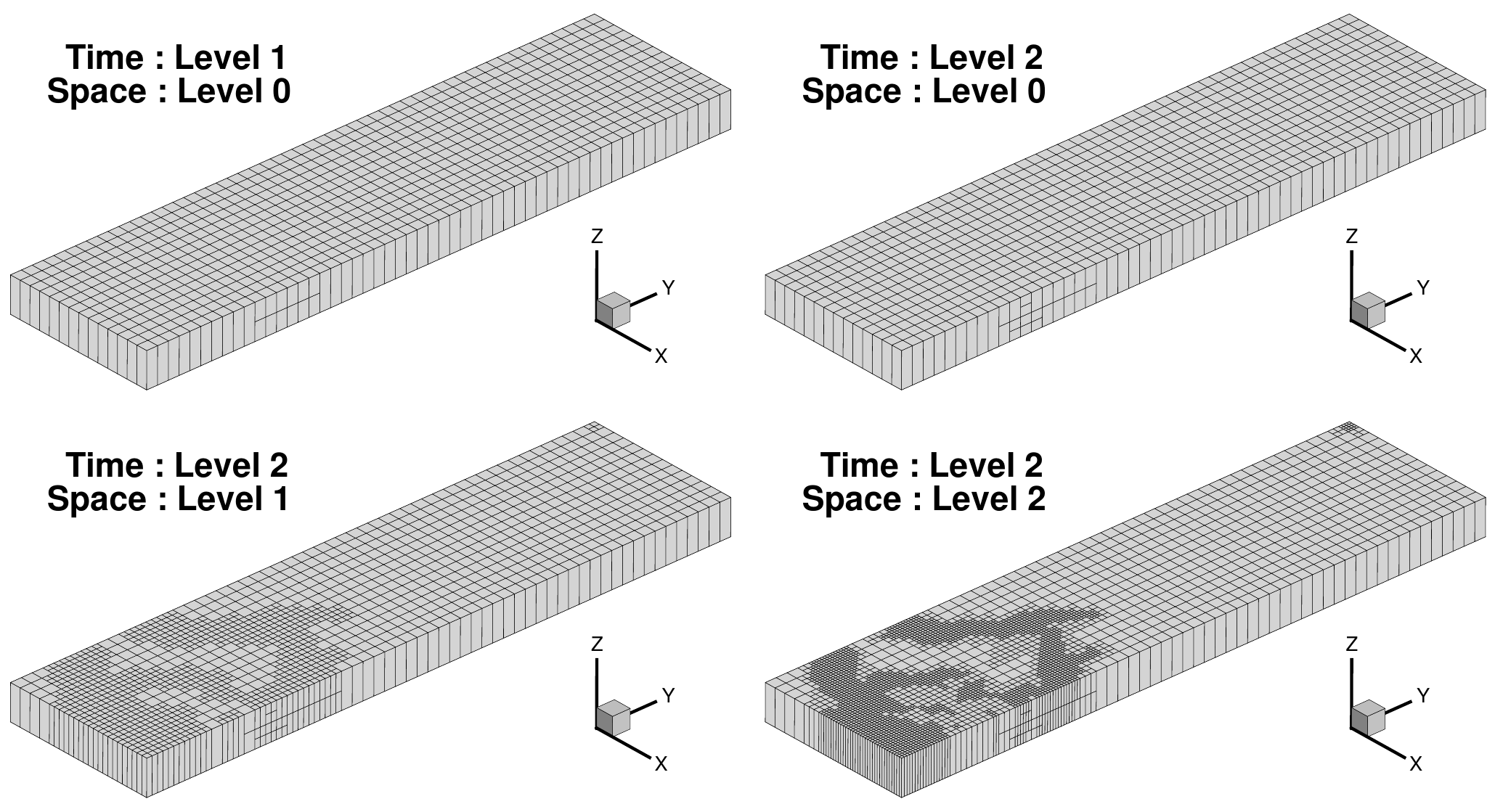}}
  \setlength{\abovecaptionskip}{-10pt}
  \caption{\bf{Sample grid generated during sequential local mesh refinement from coarsest to finest space-time resolution}}
  \label{fig:grid}
\end{figure}
\par We use the error estimators defined in Section \ref{sec:err} to scan regions for refinement that in return diminishes the upper bound of the error measure. First we study the spatial and temporal residual estimators defined by Eqn.\eqref{eq:sres} and \eqref{eq:tres}, similar to \cite{Singh:0918}. Provided by linear projection of the solution on the previous mesh, the initial guess of unknowns after each refinement procedure is naturally close to the true solution. Consequently, the residual estimators is only useful for indicating the main refinement region on the coarsest space-time resolution. Large estimator values appears only sporadically on all other resolutions as demonstrated by Fig.\ref{fig:resid}. The sporadic appearance infers strong heterogeneity of the underlying petrophysical properties. 
\begin{figure}[!]
  \center{\includegraphics[width=\textwidth]{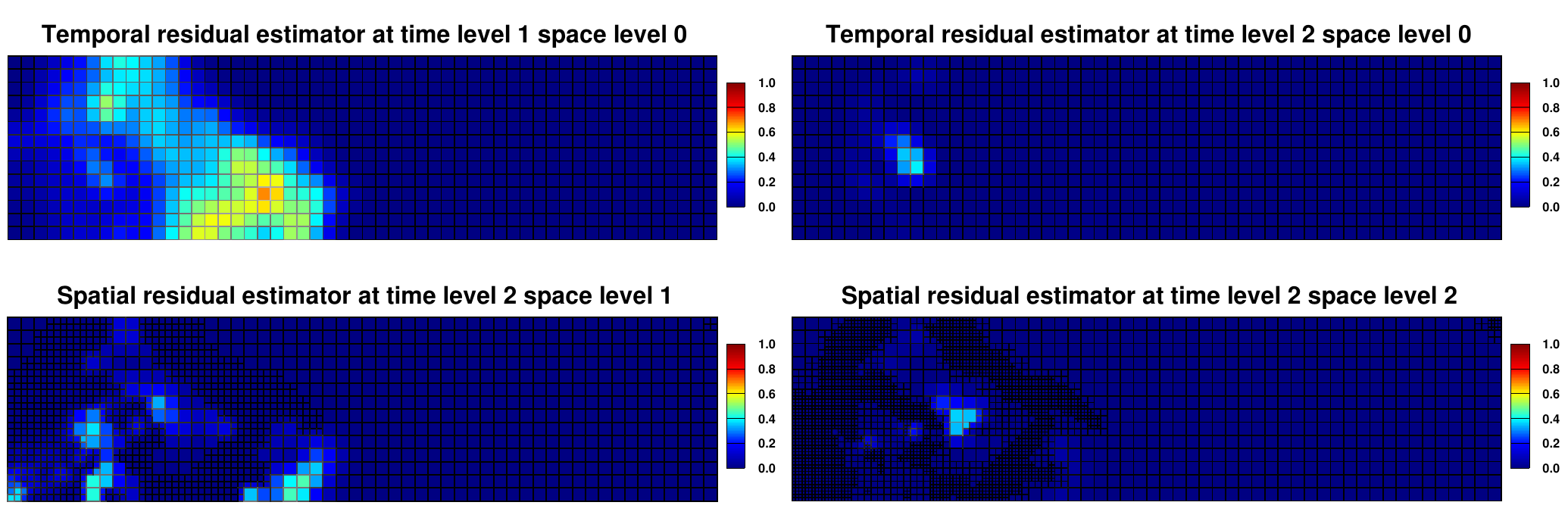}}
  \setlength{\abovecaptionskip}{-10pt}
  \caption{\bf{Normalized spatial and temporal residual estimator at each space and time refinement level}}
  \label{fig:resid}
\end{figure}
\begin{figure}[t]
  \center{\includegraphics[width=\textwidth]{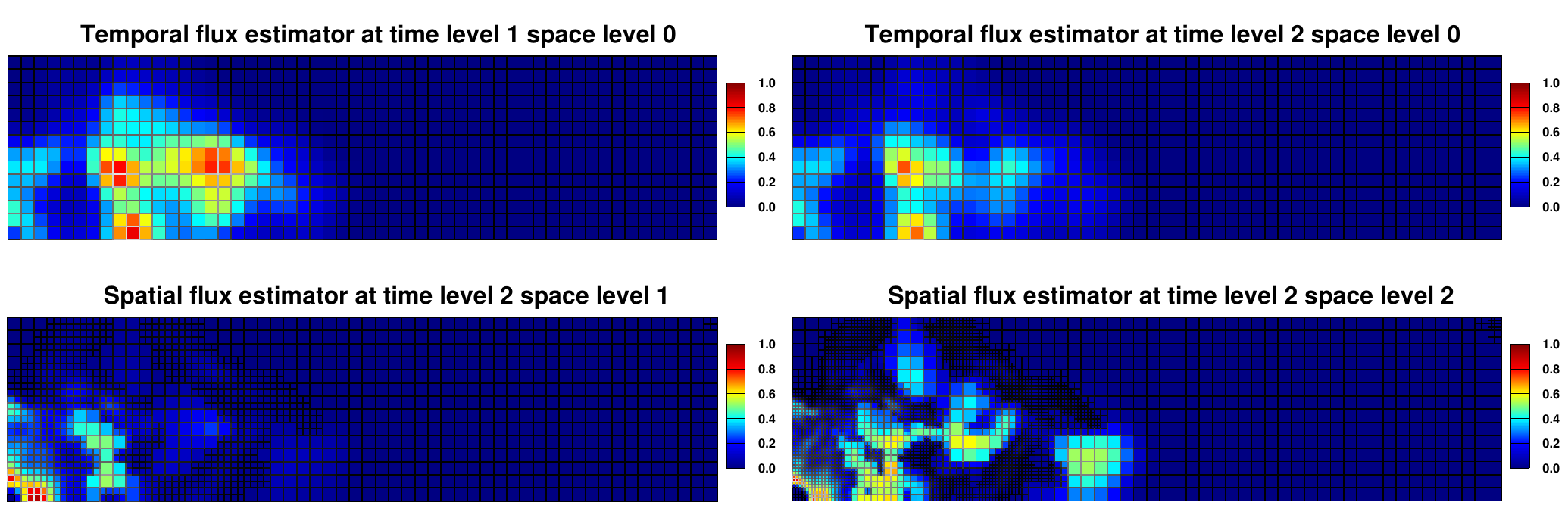}}
  \setlength{\abovecaptionskip}{-10pt}
  \caption{\bf{Normalized flux estimator at each space and time refinement level}}
  \label{fig:fluxest}
\end{figure}
\begin{figure}[!]
  \center{\includegraphics[width=\textwidth]{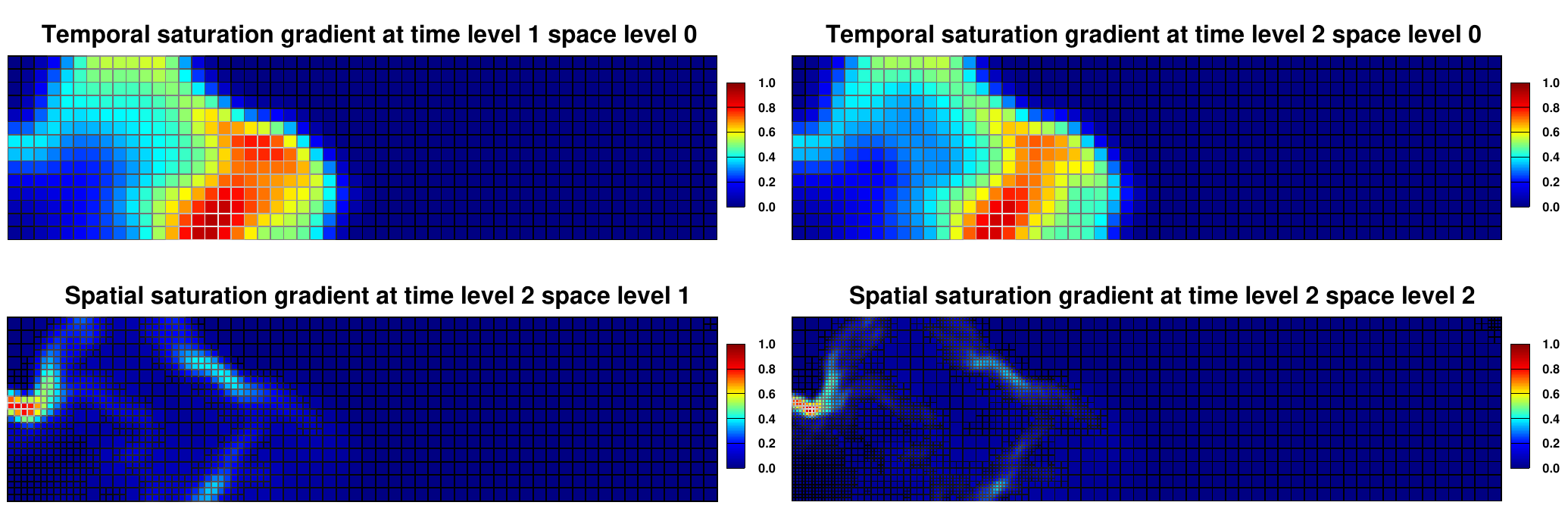}}
  \setlength{\abovecaptionskip}{-10pt}
  \caption{\bf{Normalized saturation gradient at each space and time refinement level}}
  \label{fig:errind}
\end{figure}
\par The most important estimators to ensure solution convergence and accuracy are the flux estimators. Let us first review the temporal one represented by Eqn.\eqref{eq:tflux}. We expand the original formulation as the following.
\begin{equation*}
  \eta_{t,f,\alpha,i}^{n,m}=\Bigg\{\int_{E_{i}^{m}}K^{-1}\bigg|\frac{\bs{u}_{\alpha,h}|_{\tau_{m+1}}-\bs{u}_{\alpha,h}|_{\tau_m}}{\tau_{m+1}-\tau_m}(\tau_{m+1}-t)\bigg|\Bigg\}^{\frac{1}{2}}=\Bigg\{\big|\Omega_i\big|K^{-1}\Big|\bs{u}_{\alpha,h}\big|_{\tau_{m+1}}-\bs{u}_{\alpha,h}\big|_{\tau_m}\Big|^2\ \frac{\tau_{m+1}-\tau_m}{3}\Bigg\}^{\frac{1}{2}}
\end{equation*}
Then the output is mainly controlled by the temporal flux difference term $\bs{u}_{\alpha,h}\big|_{\tau_{m+1}}-\bs{u}_{\alpha,h}\big|_{\tau_m}$, which we can further expand to
\begin{equation*}
  \begin{split}
    \bs{u}_{\alpha,h}\big|_{\tau_{m+1}}-\bs{u}_{\alpha,h}\big|_{\tau_m}=&-K\rho_{\alpha}\Big(p_{\alpha,h}\big|_{\tau_{m+1}}\Big)\frac{k_{r\alpha}\Big(s_{\alpha,h}\big|_{\tau_{m+1}}\Big)-k_{r\alpha}\Big(s_{\alpha,h}\big|_{\tau_m}\Big)}{\mu_{\alpha}}\nabla p_{\alpha,h}\big|_{\tau_{m+1}} \\
    &-K\frac{k_{r\alpha}\Big(s_{\alpha,h}\big|_{\tau_m}\Big)}{\mu_{\alpha}}\bigg[\rho_{\alpha}\Big(p_{\alpha,h}\big|_{\tau_{m+1}}\Big)\nabla p_{\alpha,h}\big|_{\tau_{m+1}}-\rho_{\alpha}\Big(p_{\alpha,h}\big|_{\tau_m}\Big)\nabla p_{\alpha,h}\big|_{\tau_m}\bigg]
  \end{split}
\end{equation*}
The second term in the above equation is effectively zero in slightly compressible flow since density variation caused by pressure is negligible and pressure gradient stays fairly constant in time. Moreover, temporal refinement is not necessary if large estimator output is caused by the leading constant $\frac{K}{\mu_{\alpha}}\rho_{\alpha}\big(p_{\alpha,h}|_{\tau_{m+1}}\big)\nabla p_{\alpha,h}|_{\tau_{m+1}}$ in the first term (eg. regions around the well with large pressure gradient), since pressure solution is smooth in time and does not trigger any convergence issues. We need to apply temporal refinement in regions with large $\eta_{t,f,\alpha,i}^{n,m}$ caused specifically by significant change in relative permeability $k_{r\alpha}\big(s_{\alpha,h}|_{\tau_{m+1}}\big)-k_{r\alpha}\big(s_{\alpha,h}|_{\tau_m}\big)$. Therefore we calculate the temporal water saturation gradient
\begin{equation}
  \varepsilon_{t,i}^{n,m}=\Bigg|\frac{\partial}{\partial t}s_{w,h,i}^{n,m}\Bigg|
  \label{eq:trefind}
\end{equation}
and applied refinement exclusively to regions with both $\eta_{t,f,\alpha,i}^{n,m}$ and $\varepsilon_{t,i}^{n,m}$ values exceeding the threshold.
\par Similarly, the spatial flux estimator $\eta_{s,f,\alpha,i}^{n,m}$ in Eqn.\eqref{eq:sflux} can be expanded to
\begin{equation*}
  \eta_{s,f,\alpha,i}^{n,m}=\bigg\{\big|\Omega_i\big|\ \big|\tau_{m+1}-\tau_m\big|K^{-1}\big|\bs{u}_{up,\alpha,h}-\bs{u}_{\alpha,h}\big|^2\bigg\}^{\frac{1}{2}}
\end{equation*}
with the output mainly controlled by the flux spatial difference $\bs{u}_{up,\alpha,h}-\bs{u}_{\alpha,h}$. We can also expand this term as the following.
\begin{equation*}
  \bs{u}_{up,\alpha,h}-\bs{u}_{\alpha,h}=-K\rho_{\alpha}\frac{k_{r\alpha}\big(s_{up,\alpha,h}\big)-k_{r\alpha}\big(s_{\alpha,h}\big)}{\mu_{\alpha}}\nabla p_{\alpha,h}
\end{equation*}
Surely, we need to refine regions with significant change in relative permeability $k_{r\alpha}\big(s_{up,\alpha,h}\big)-k_{r\alpha}\big(s_{\alpha,h}\big)$ to accurately represent the features of the reservoir. Furthermore regions with large estimator output caused by the leading constant $\frac{K}{\mu_{\alpha}}\rho_{\alpha}\nabla p_{\alpha,h}$ also need special care. Such regions are characterized by rapid mass flow and refining them facilitates convergence. Therefore we calculate the spatial water saturation gradient.
\begin{equation}
  \varepsilon_{s,i}^{n,m}=
  \begin{cases}
  \begin{aligned}
    &\big\|\nabla s_{w,h,i}^{n,m}(\bs{x_{i}})\big\|_{l^{\infty}} &\quad if\ \big\|\nabla s_{w,h,i}^{n,m}(\bs{x_i})\big\|_{l^{\infty}}>\big\|\nabla s_{w,h,i}^{n-1,m}(\bs{x_i})\big\|_{l^{\infty}}  \\
    &\frac{1}{2}\Big(\big\|\nabla s_{w,h,i}^{n,m}(\bs{x_i})\big\|_{l^{\infty}}+\big\|\nabla s_{w,h,i}^{n-1,m}(\bs{x_i})\big\|_{l^{\infty}}\Big) &\quad if\ \big\|\nabla s_{w,h,i}^{n,m}(\bs{x_i})\big\|_{l^{\infty}}\leq\big\|\nabla s_{w,h,i}^{n-1,m}(\bs{x_i})\big\|_{l^{\infty}}
  \end{aligned}
  \end{cases}
  \label{eq:srefind}
\end{equation}
and apply refinement to regions with either $\eta_{s,f,\alpha,i}^{n,m}$ or $\varepsilon_{s,i}^{n,m}$ values exceeding the threshold. Please note that we are taking some extra steps when calculating spatial saturation gradient by looking at the previous time step values. This mechanism ensures a more accurate exposure of features in the system, especially in channelized permeability distributions. Fig.\ref{fig:fluxest} and \ref{fig:errind} show flux estimator and saturation gradient at each refinement level.
\begin{figure}[!]
  \center{\includegraphics[width=\textwidth]{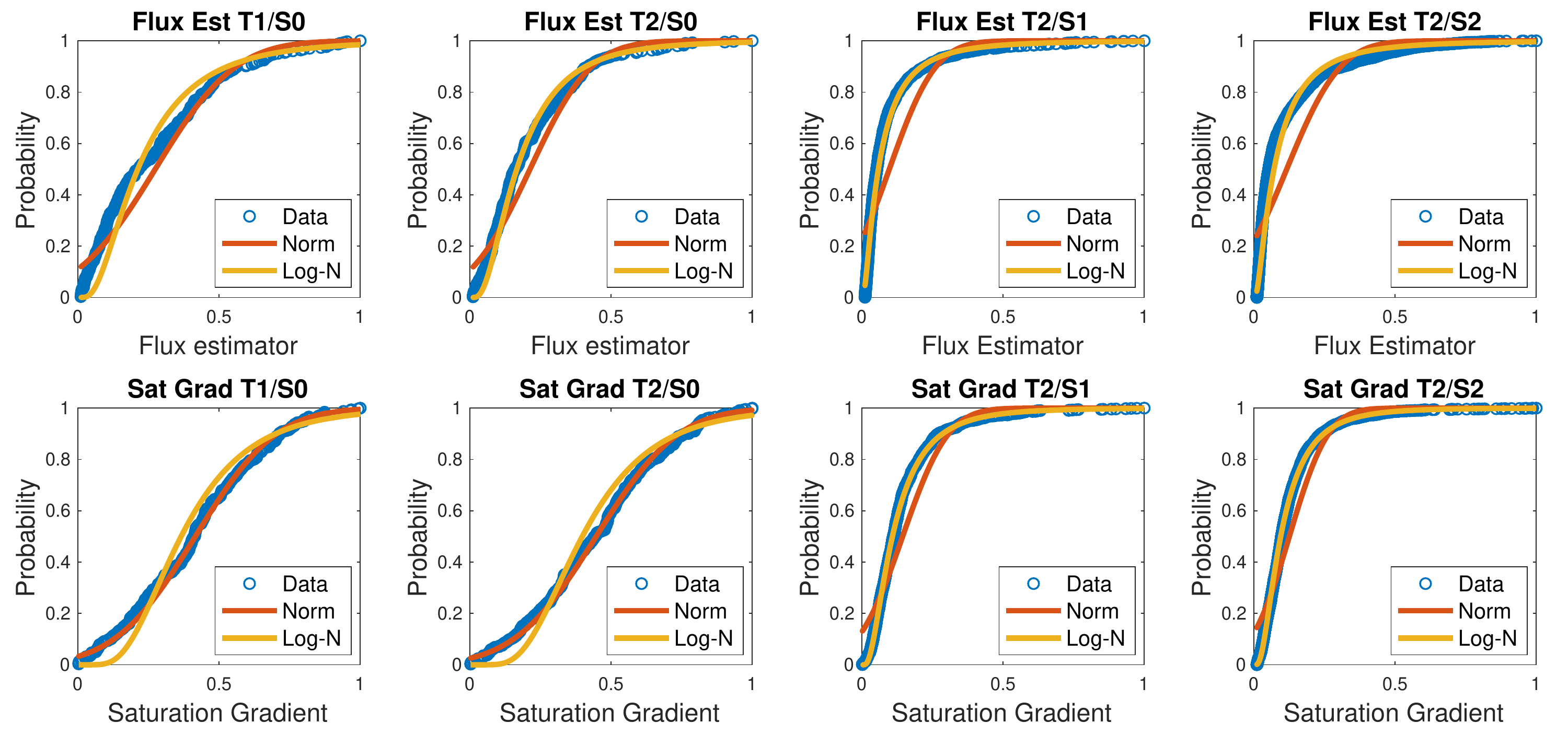}}
  \setlength{\abovecaptionskip}{-10pt}
  \caption{\bf{Cumulative distribution function fitted to flux estimator (top 4 plots) and saturation gradient (bottom 4 plots) data at each refinement level}}
  \label{fig:cdfer}
\end{figure}
\par Instead of setting a subjective threshold, we outline the regions for refinement by distribution percentile. We first define $[0.01,1]$ as the analysis range of flux error estimator and saturation gradient. Values below $0.01$ are considered too small and thus neglected. The threshold is determined by distribution. The cumulative distribution function of flux error estimator and saturation gradient at each refinement level is plotted in Fig.\ref{fig:cdfer} against sample data collected during simulation. As illustrated by the graphs, the data for both variables generally follow log-normal distribution trend. During temporal adaptation, we refine cells with $50\%$ largest values in both flux estimator and saturation gradient. Therefore, we use the log-mean, which covers approximately $50\%$ of the analysis range, as the threshold. We notice that during temporal refinement, the cumulative distribution functions of both variables are better described by normal distribution. However, we still choose the log-mean as the threshold since it leads to a slight over-refinement in time and thus better guarantees Newton convergence. During spatial adaptation, we refine cells with either $50\%$ largest values in the saturation gradient or $10\%$ largest values in flux estimator. Therefore, the thresholds for the two variables are the log-mean and one standard deviation above the log-mean, corresponding to their respective distribution.
\par Finally for the non-conformity estimators, $\eta_{t,p,\alpha,i}^{n,m}$ is naturally diminished during temporal refinement using $\eta_{t,f,\alpha,i}^{n,m}$ since the two estimators have similar formulations. $\eta_{s,p,\alpha,i}^{n,m}$ represents the tangential gradient of flux on non-conformal grid interfaces. To reduce this term, we apply mesh smoothing algorithm. Adjacent grid cells cannot be more than one refinement level apart within the hierarchical tree structure in both spatial and temporal dimensions. Such algorithm also ensures a smooth transition from fine grid into coarse grid and thus facilitates convergence.

\section{Numerical results}
\label{sec:num}

\begin{figure}[t]
  \center{\includegraphics[width=\textwidth]{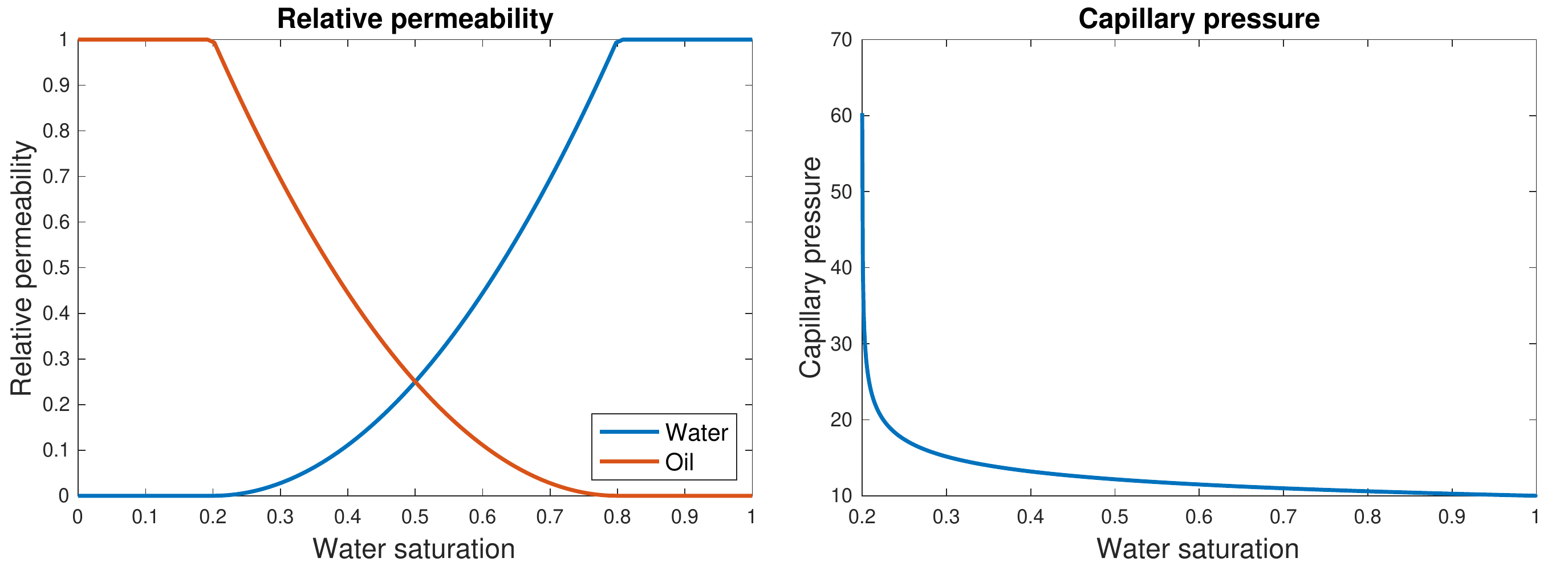}}
  \setlength{\abovecaptionskip}{-10pt}
  \caption{\bf{Relative permeability (left) and capillary pressure (right) curve for numerical experiments}}
  \label{fig:brooks}
\end{figure}
\par In this section we will show results from two numerical experiments on 2-D two-phase flow model. Both experiments use the same fluid data from the SPE10 dataset \cite{Christie:0801}. The oil and water reference densities in Eqn.\eqref{eq:comp} are taken to be $53\ lb/ft^{3}$ and $64\ lb/ft^{3}$ and compressibilities are $1\times 10^{-4}\ psi^{-1}$ and $3\times 10^{-6}\ psi^{-1}$ respectively. We use Brooks's Corey model for both relative permeability and capillary pressure. The equations for relative permeability are
\begin{equation}
  \begin{cases}
  \begin{aligned}
    &k_{rw}=k_{rw}^{0}\bigg(\frac{s_{w}-s_{wirr}}{1-s_{or}-s_{wirr}}\bigg)^{n_{w}} \\
    &k_{ro}=k_{ro}^{0}\bigg(\frac{s_{o}-s_{or}}{1-s_{or}-s_{wirr}}\bigg)^{n_{o}}
  \end{aligned}
  \end{cases}
  \label{eq:bcrel}
\end{equation}
The endpoint values are $s_{or}=s_{wirr}=0.2$ and $k_{ro}^{0}=k_{rw}^{0}=1.0$ while the model exponents are $n_{w}=n_{o}=2$. The equation for capillary pressure is
\begin{equation}
  p_{c}(s_{w})=P_{en,cow}\bigg(\frac{1-s_{wirr}}{s_{w}-s_{wirr}}\bigg)^{l_{cow}}
  \label{eq:bccap}
\end{equation}
with $P_{en,cow}=10\ psi$ and $l_{cow}=0.2$. Fig.\ref{fig:brooks} visualizes the relative permeability and capillary pressure curve. The two experiments uses gaussian-like and channelized permeability and porosity distributions from SPE 10 dataset \cite{Christie:0801} layer 20 and 52, respectively. The reservoir size is $56ft \times 216ft \times 1ft$. We place a water rate specified injection well at the bottom left corner and a pressure specified production well at the upper right corner. The water injection rate is $1\ ft^{3}/day$ and production pressure is $1000\ psi$. Furthermore, the initial pressure and water saturation are set to be $1000\ psi$ and $0.2$.

\subsection{Gaussian-like Permeability Distribution}
\label{subsec:gau}

\par The gaussian-like permeability field comes from SPE 10 dataset layer 20. The fine scale petrophysical data are shown in Fig.\ref{fig:gaup}, assuming isotropic permeability. We allow three refinement levels in both space and time in our experiment. Although the framework allows different refinement ratios between levels, for the sake of simplicity we set the same ratio, a factor of 2, between all levels. We use the numerical homogenization technique introduced in \cite{Amanbek:0619} to upscale the fine scale permeability to different coarse levels. This calculation only needs to be performed once at the beginning of the experiment. The homogenized permeability distribution in $X$ and $Y$ directions, which does not manifest high anisotropy, is illustrated in Fig.\ref{fig:gauh}. The porosity is upscaled simply by weighted volumetric average and therefore is not visualized. The computational domain is $56\ ft\times 216\ ft\times 1\ ft\times 1000\ days$ with coarsest and finest element size of $8\ ft\times 8\ ft\times 1\ ft\times 10\ days$ and $1\ ft\times 1\ ft\times 1\ ft\times 1.25\ days$.
\begin{figure}[t]
  \center{\includegraphics[width=\textwidth]{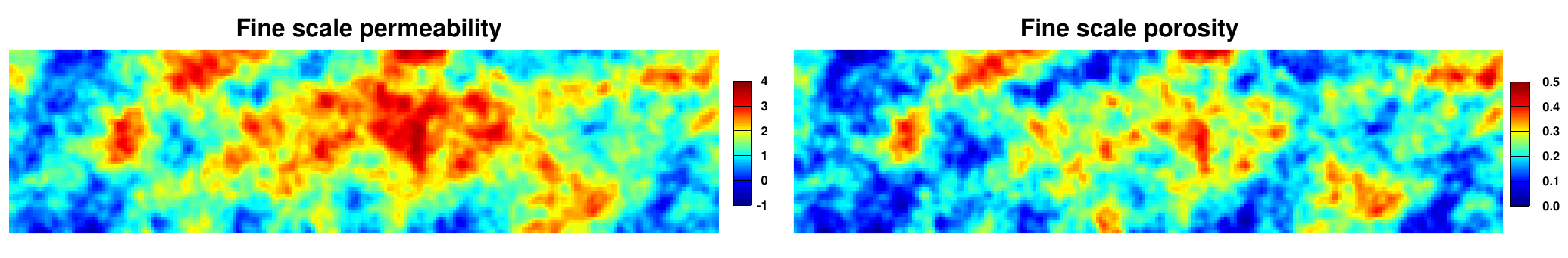}}
  \setlength{\abovecaptionskip}{-10pt}
  \caption{\bf{Gaussian-like fine scale permeability (left) and porosity (right) distribution}}
  \label{fig:gaup}
\end{figure}
\begin{figure}[!]
  \center{\includegraphics[width=\textwidth]{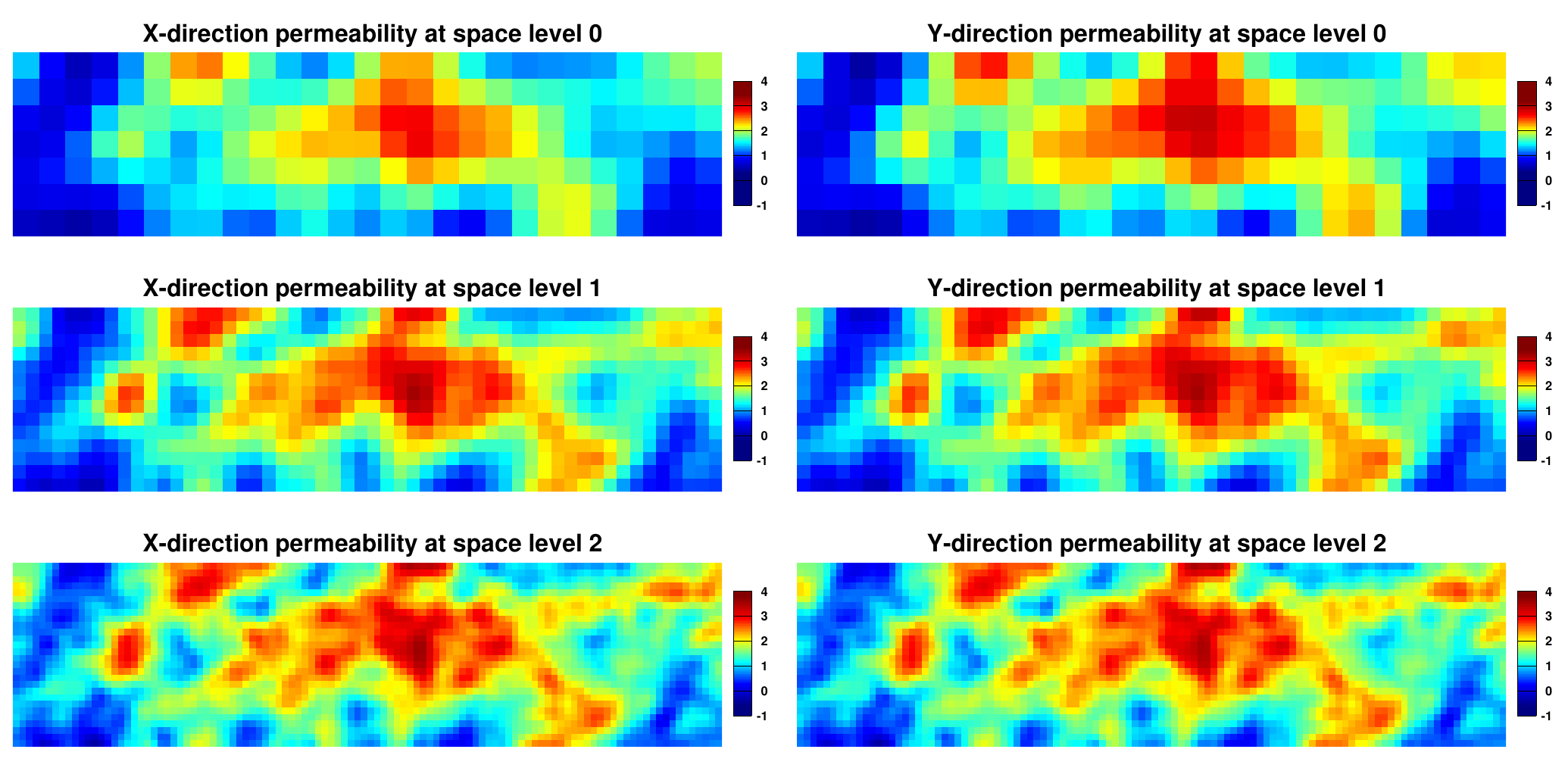}}
  \setlength{\abovecaptionskip}{-10pt}
  \caption{\bf{Homogenized gaussian-like permeability in $X$ and $Y$ direction for each space level}}
  \label{fig:gauh}
\end{figure}
\par The adaptive water saturation profile with its mesh as compared to fine scale solution at 100 and 500 days are plotted in Fig.\ref{fig:gaure}. We observe the finest mesh stays concentrated at the water front to correctly capture the dramatic changes in saturation. In this region, mass transfer is not dominated by either oil or water phase and thus contributes the most non-linearity and requires temporal refinement for stable Newton convergence. Elements behind the water front is gradually coarsened due to the decreased saturation variation. Overall, the saturation profile provided by the sequential refinement solver looks similar to the fine scale solution. Fig.\ref{fig:gaurc} shows the production rates and cumulative recoveries of the two solutions, which are nearly identical. The oil rate from the sequential refinement solver appear to be slightly smoother at the early time and has tiny oscillation around water break though, which is caused by the coarse mesh.
\begin{figure}[t]
  \center{\includegraphics[width=\textwidth]{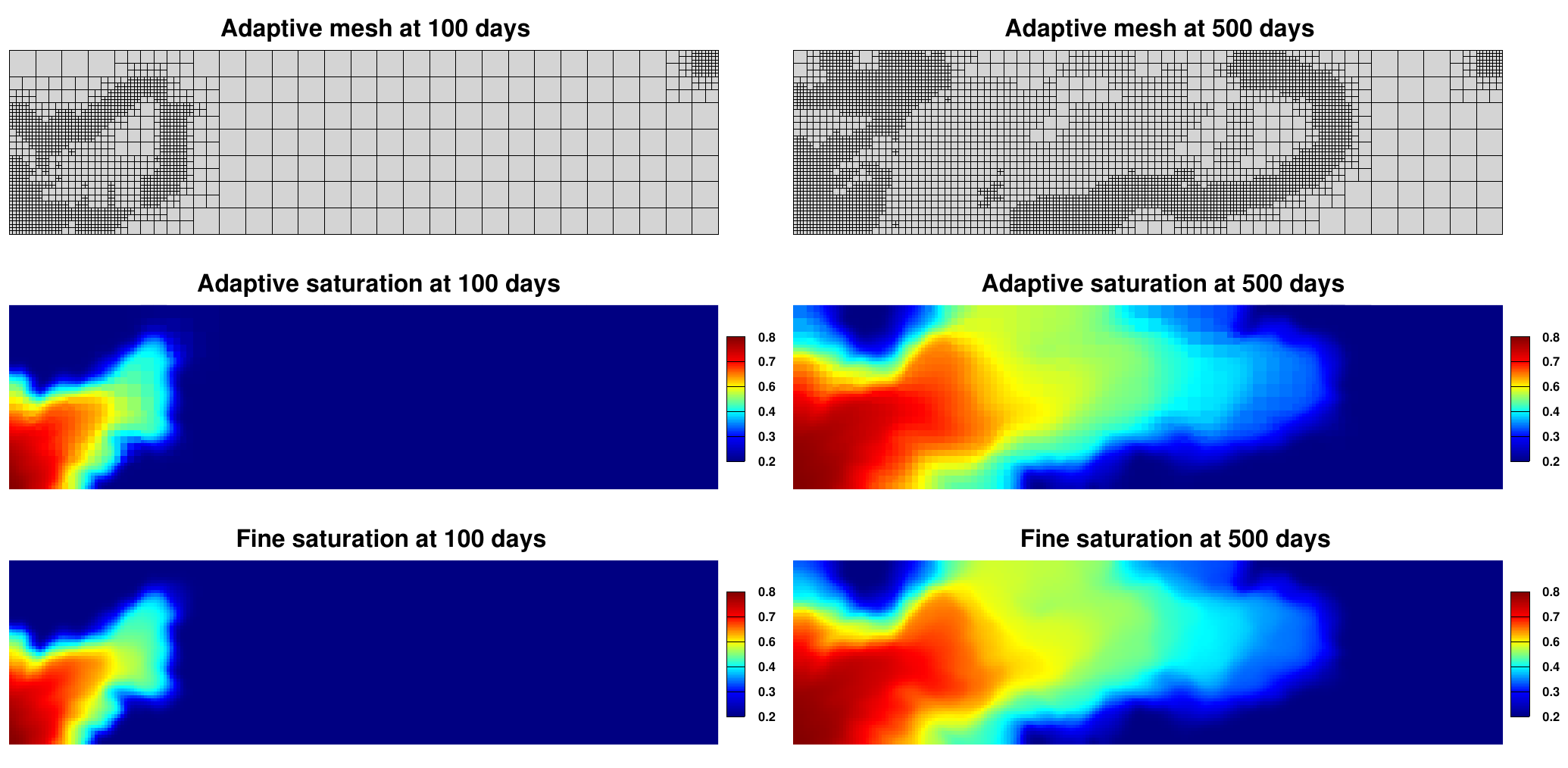}}
  \setlength{\abovecaptionskip}{-10pt}
  \caption{\bf{Adaptive mesh (top) and water saturation profile(middle) generated by sequential refinement solver as compared to fine scale solution (bottom) at 100 and 500 days in gaussian-like permeability field}}
  \label{fig:gaure}
\end{figure}
\begin{figure}[!]
  \center{\includegraphics[width=\textwidth]{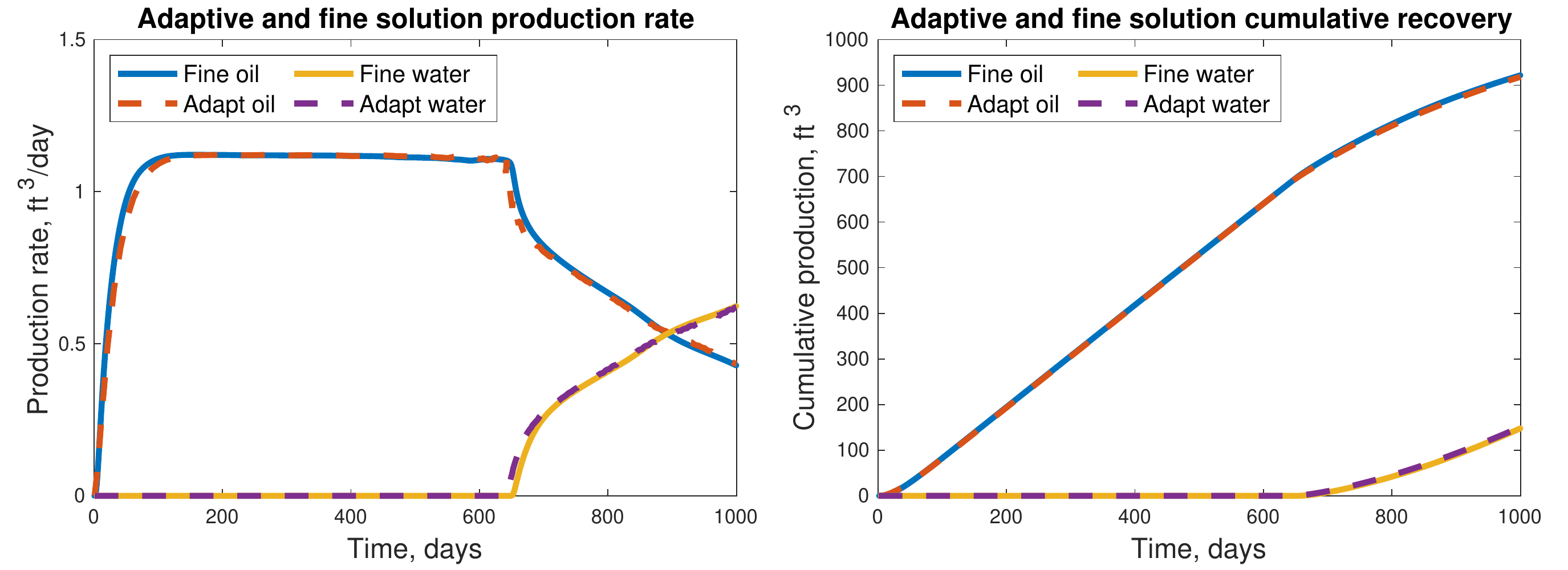}}
  \setlength{\abovecaptionskip}{-10pt}
  \caption{\bf{Two phase production rates and cumulative recoveries from adaptive and fine scale solution of gaussian-like permeability field}}
  \label{fig:gaurc}
\end{figure}
\par The program execution time is presented in Fig.\ref{fig:rtg}. The total execution time consists of system setup which constructs the linear system, solving the linear system and data handle which mainly involves copying and pasting data from the current to the previous time step. Since the experiment problem size is still small, we use both direct and iterative solver to resolve the linear system. The semi-structured space-time mesh results in highly non-symmetric matrices and therefore we use GMRES with ILU preconditioner as our iterative solver. We observe 8 and 4 times speedup on system setup and data handle using direct solver. These two types of operations are highly dependent upon the number of time steps taken and total number of refinement levels. Hence, the speedup scales linearly with the total temporal refinement ratio and the same runtime reduction behavior is observed when using iterative solver. The speedup on solving the linear system best represents the computational performance improvement. Since our problem size is small, the efficiency gain is not substantial when using direct solver. On the contrary, we observe 25 times speed up on solving the linear system when using iterative solver. Additional techniques on solving non-symmetric linear systems iteratively, such as relaxing linear solver tolerance using forcing function \cite{Eisenstat:0196, Lacroix:0403} and applying specialized preconditioners \cite{France:0219, Dawson:0997} for Krylov-based method, may be utilized for additional acceleration. Note that as we move towards more complex models such as 3-D black oil, the solution to the corresponding linear system is only accessible through iterative methods. Thus we should expect significant improvement on computational efficiency once we approach those types of problems.
\begin{figure}[!]
  \center{\includegraphics[width=\textwidth]{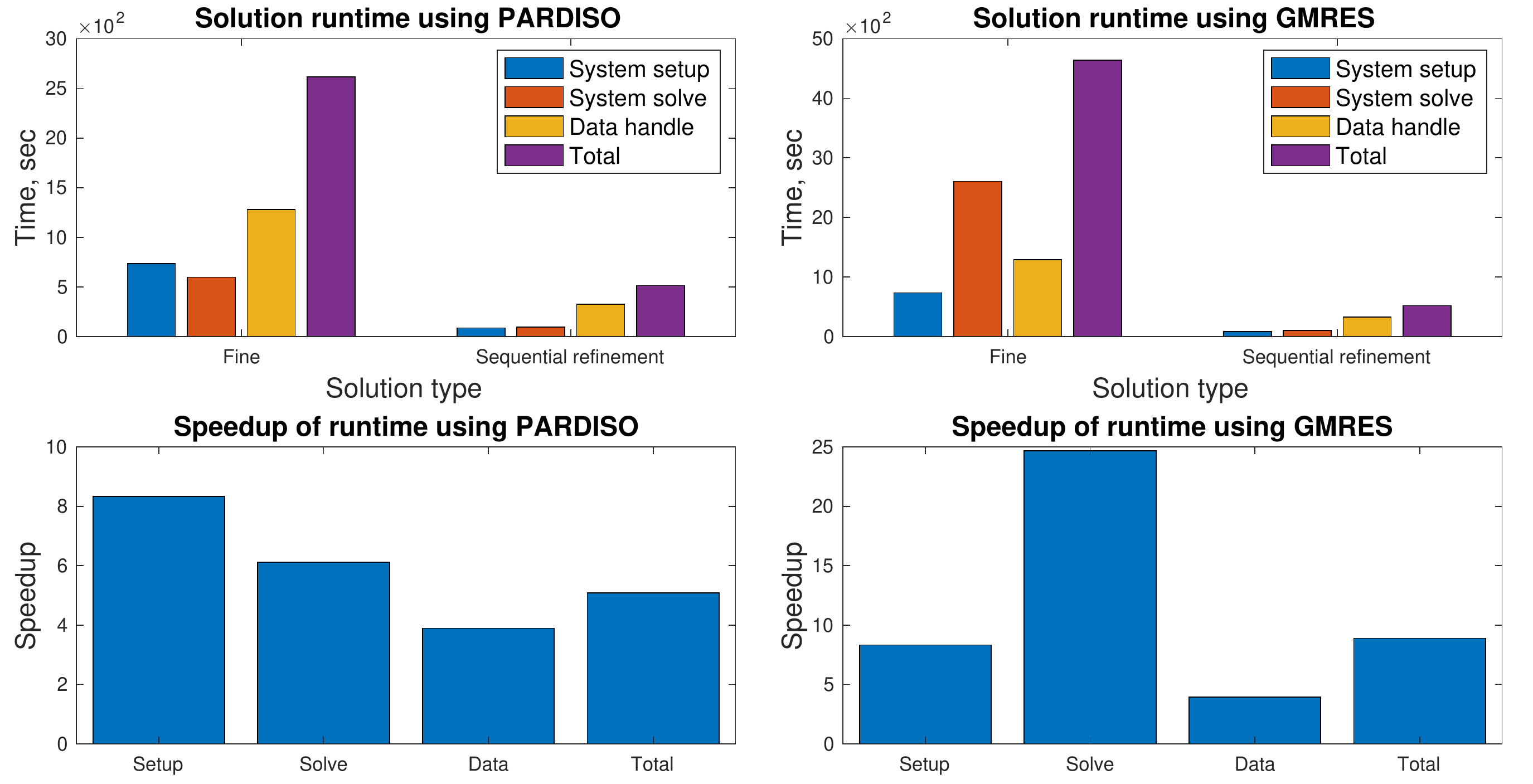}}
  \setlength{\abovecaptionskip}{-10pt}
  \caption{\bf{Runtime comparison between sequential refinement and fine scale solution using direct and iterative linear solver for gaussian-like permeability field}}
  \label{fig:rtg}
\end{figure}

\subsection{Channelized Permeability Distribution}
\label{subsec:chan}

\begin{figure}[t]
  \center{\includegraphics[width=\textwidth]{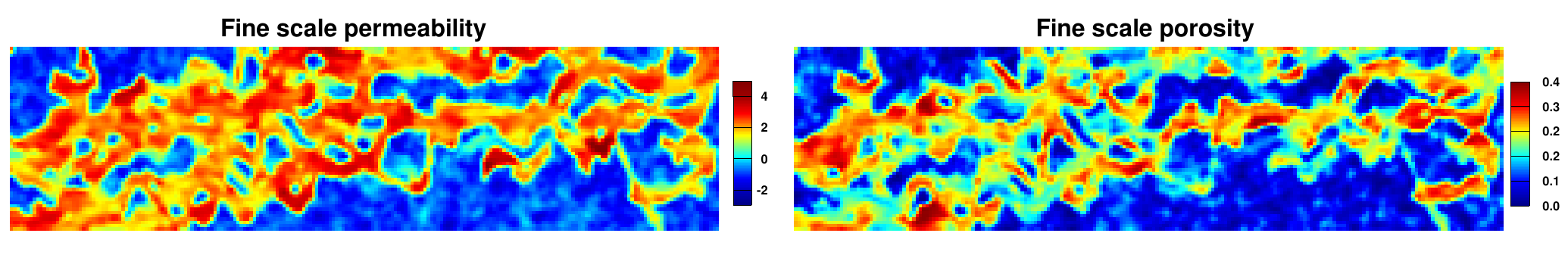}}
  \setlength{\abovecaptionskip}{-10pt}
  \caption{\bf{Channelized fine scale permeability (left) and porosity (right) distribution}}
  \label{fig:chap}
\end{figure}
\begin{figure}[!]
  \center{\includegraphics[width=\textwidth]{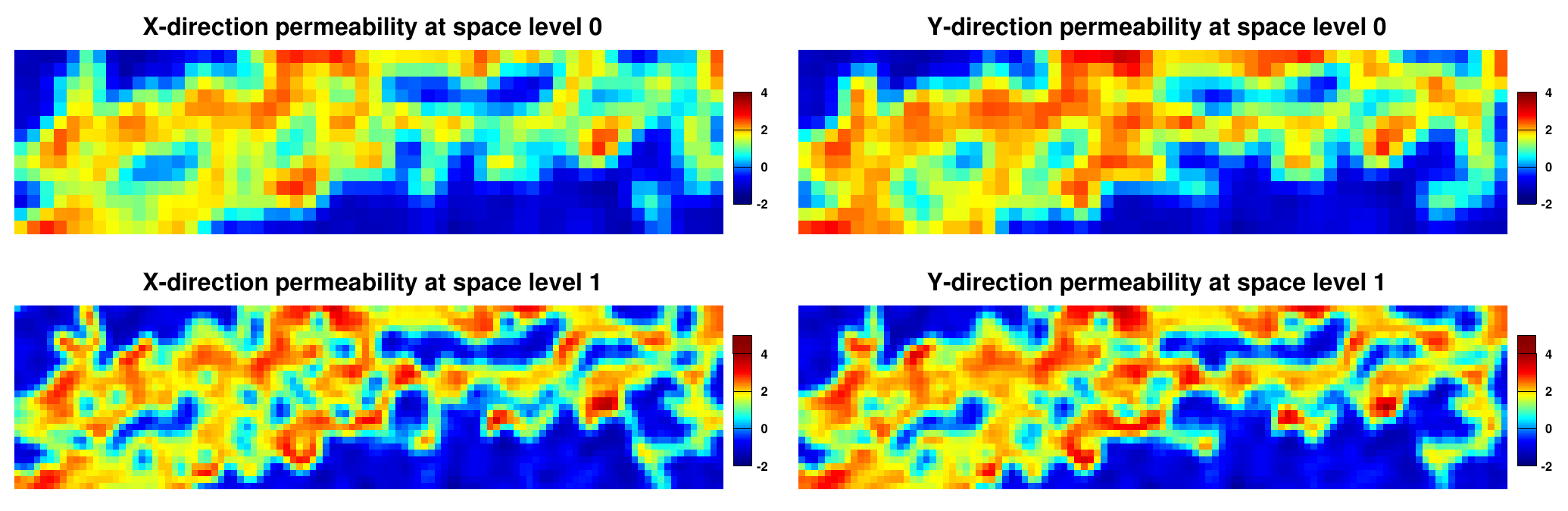}}
  \setlength{\abovecaptionskip}{-10pt}
  \caption{\bf{Homogenized channel permeability in $X$ and $Y$ direction for each space level}}
  \label{fig:chah}
\end{figure}
\par The channelized permeability field comes from SPE 10 dataset layer 52. The fine scale petrophysical data are shown in Fig.\ref{fig:chap}. Unlike the gaussian case, the permeability and porosity here is highly structured and some structures are subjected to destruction during the homogenization process. We allow two refinement levels in both space and time for this experiment. The refinement ratio is also set to a factor of 2 between all levels. During numerical homogenization, we impose oversampling technique introduced by \cite{Efendiev:09} and \cite{Chung:0916} to preserve channel connectivity as much as possible. The homogenized permeability distribution in $X$ and $Y$ direction is illustrated in Fig.\ref{fig:chah}. On the contrary to the gaussian case, the upscaled channel permeability is highly anisotropic. The computational domain is $56\ ft\times 216\ ft\times 1\ ft\times 1000\ days$ with coarsest and finest element size of $8\ ft\times 8\ ft\times 1\ ft\times 5\ days$ and $1\ ft\times 1\ ft\times 1\ ft\times 1.25\ days$.
\begin{figure}[!]
  \center{\includegraphics[width=\textwidth]{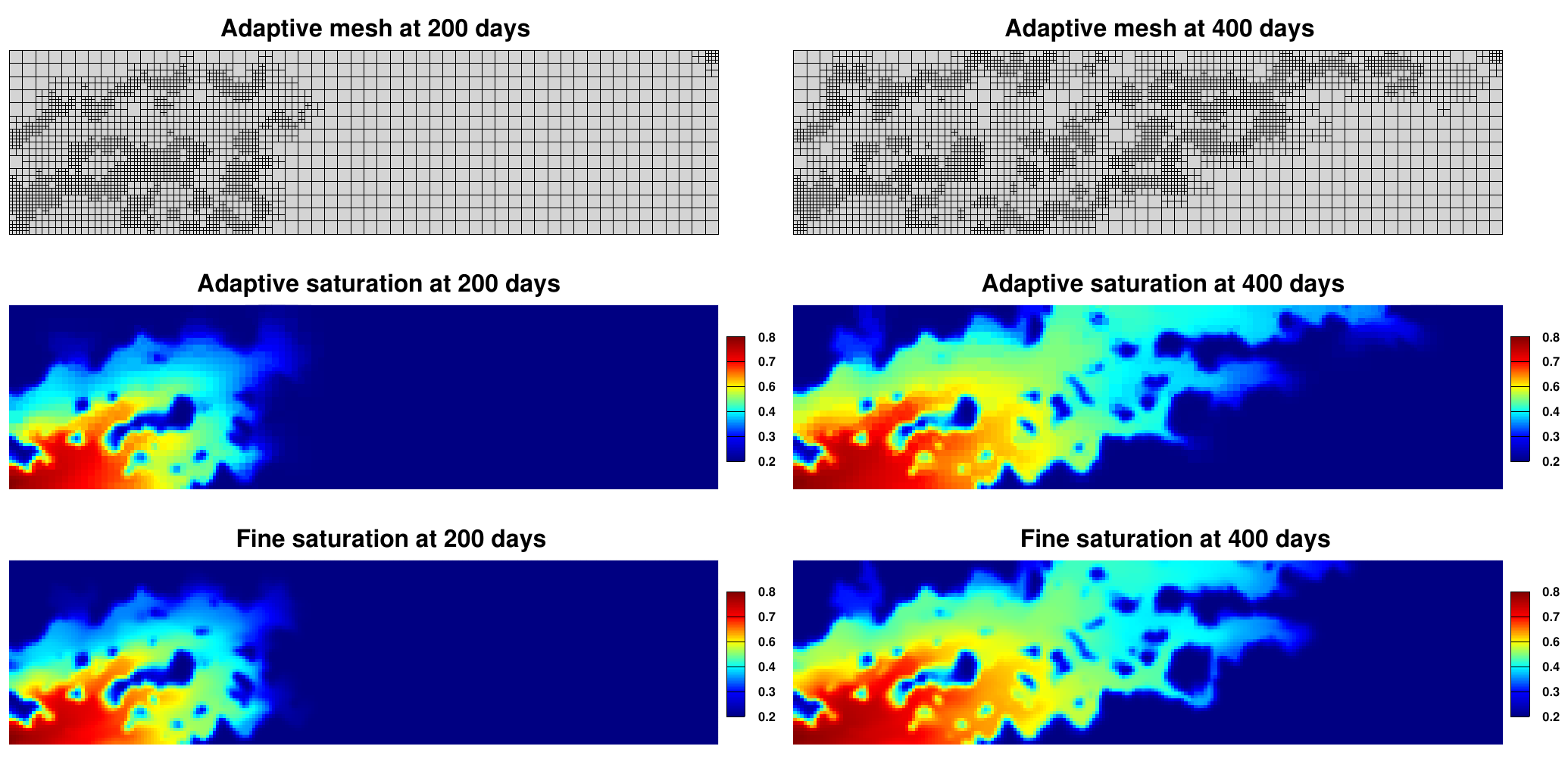}}
  \setlength{\abovecaptionskip}{-10pt}
  \caption{\bf{Adaptive mesh (top) and saturation profile(middle) generated by sequential refinement solver as compared to fine scale solution (bottom) at 200 and 400 days in channelized permeability field}}
  \label{fig:chare}
\end{figure}
\begin{figure}[t]
  \center{\includegraphics[width=\textwidth]{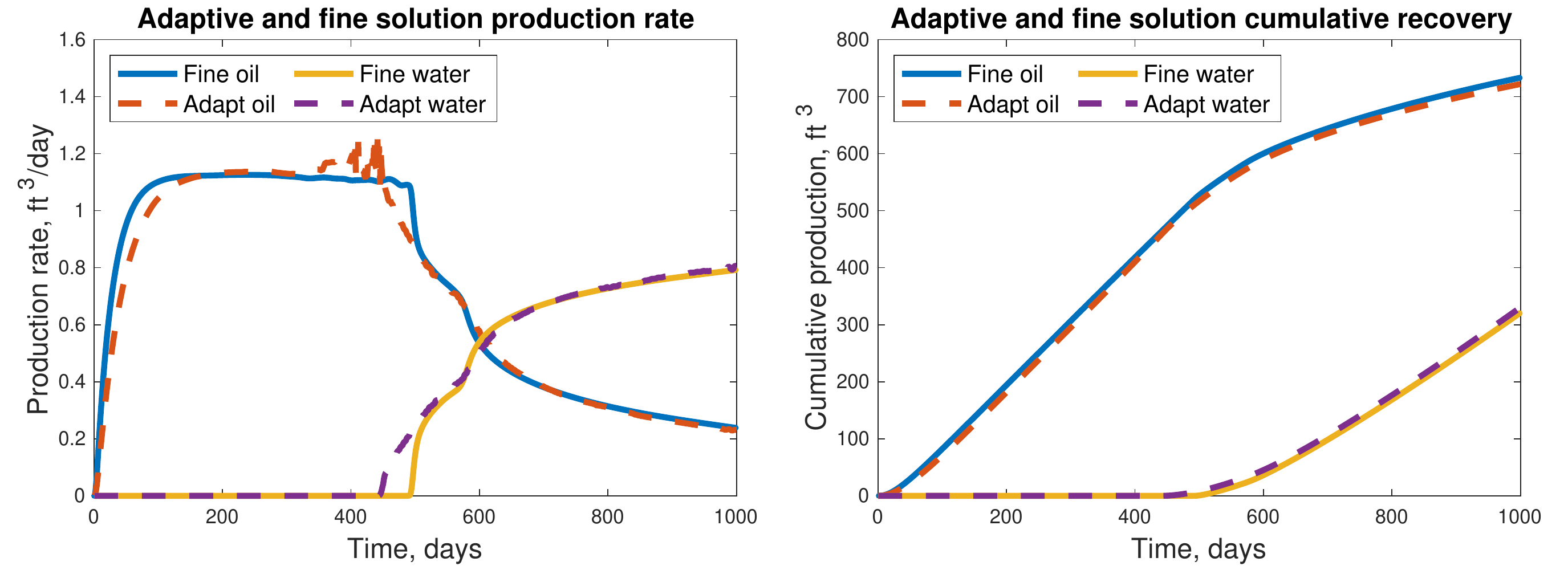}}
  \setlength{\abovecaptionskip}{-10pt}
  \caption{\bf{Two phase production rates and cumulative recoveries from adaptive and fine solution of channelized permeability field}}
  \label{fig:charc}
\end{figure}
\begin{figure}[!]
  \center{\includegraphics[width=\textwidth]{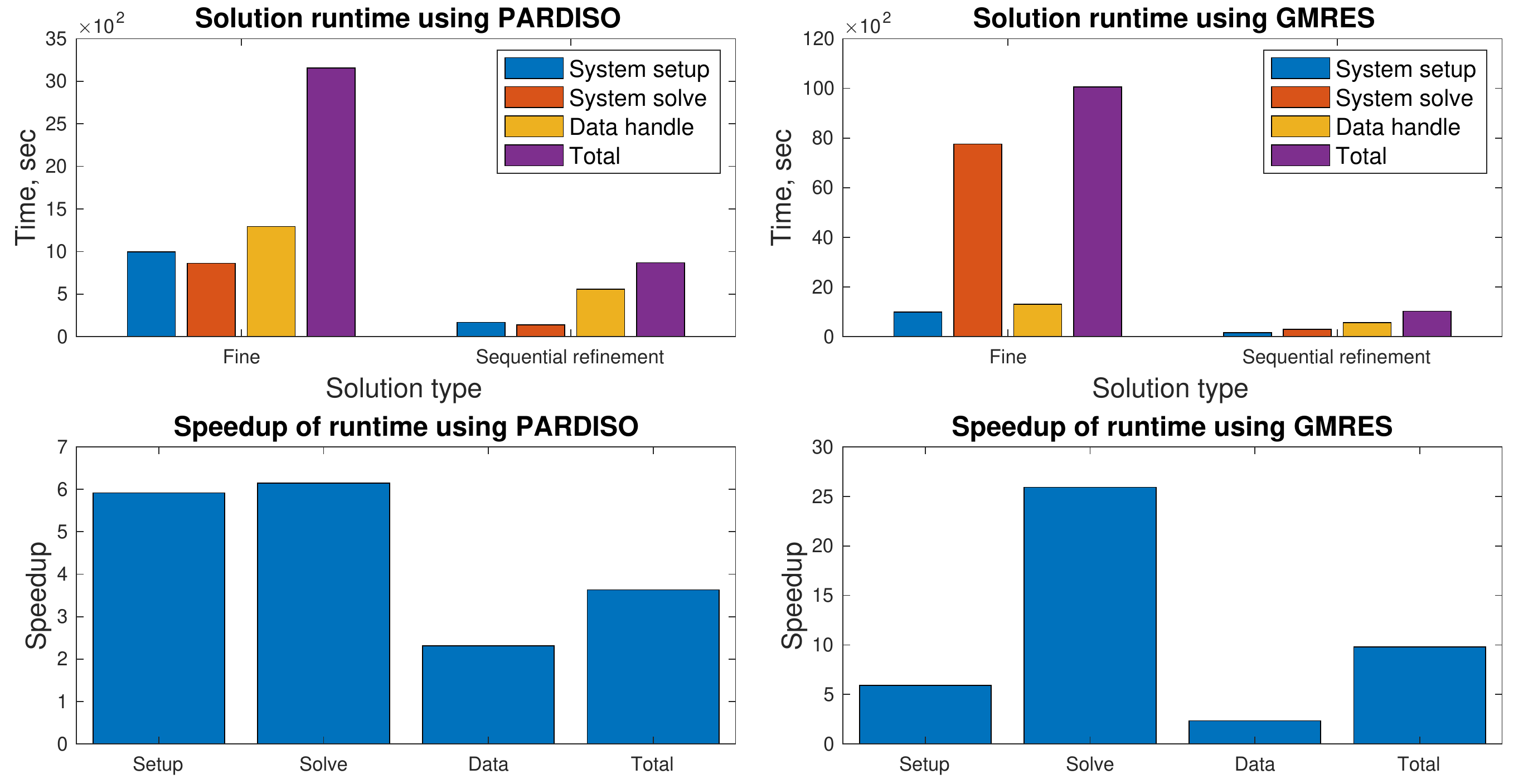}}
  \setlength{\abovecaptionskip}{-10pt}
  \caption{\bf{Runtime comparison between sequential refinement and fine scale solution using direct and iterative linear solver for channelized permeability field}}
  \label{fig:rtc}
\end{figure}
\par The adaptive water saturation profile with its mesh as compared to the fine scale solution at 200 and 400 days are plotted in Fig.\ref{fig:chare}. The overall production profile also resembles each other between the two solutions. Here, the fine mesh not only concentrates at the water front, but also outlines the channel structure. The channel boundary is characterized by dramatic contrast of permeability, thus resulting in steep water saturation gradient. The refinement algorithm detects these features and deploys mesh with appropriate size accordingly. Many low permeability spots inside the main high permeability channel are also accurately identified and represented. Fig.\ref{fig:charc} shows the production rates and cumulative recoveries of the two solutions. The adaptive and fine scale rates also look similar, however with obvious discrepancies. The rates from sequential refinement solver looks smoother than the fine scale solution at early time. It also suffers from slightly early water breakthrough. The oil and water cumulative production from the two solutions nearly overlap.
\par We also approach the solution by both direct and iterative method. The program execution time is shown in Fig.\ref{fig:rtc}. The speedup on system setup and data handle also scales linearly with total temporal refinement ratio. The solution time reduction by direct solver remains low. However, we still observe a 25 times speedup using iterative solver, even when the coarse time step is halved as compared to the gaussian case. The substantial improvement is caused by two main reasons. First of all, the flow and transport is constrained within the channel structure behind the water front, making the saturation variation effectively zero in other part of the reservoir. Consequently, the number of grid cells required to represent the channel structure and saturation front is relatively small, causing the adaptive solution easier to acquire. Secondly, the fine scale system consists of dramatic permeability contrast, resulting in the related linear system to have eigenvalues close to zero. Solving such linear system with Krylov-based iterative methods requires many iterations, making the fine scale solution harder to obtain.

\section{Conclusions}
\label{sec:con}

\par We have introduced an algorithm that constructs adaptive mesh using error estimators to solve non-linear two-phase flow problems with reduced execution time. The procedure sequentially refines the mesh from coarsest to finest resolution in large non-linearity regions, with temporal and spatial adaptivity separated to accurately expose features in the system while ensuring numerical convergence. After each refinement, the initial guess for the new mesh is generated by the solution on the previous mesh through linear projection, which accelerates convergence rate. Results from two numerical experiments are demonstrated. Rates and cumulative production from both experiments resembles well between the adaptive and fine scale solution. The water saturation profiles also look similar. We observe approximately 25 times speedup in solution time for the gaussian-like and the channelized permeability field. The channel case suffers from a slightly early water breakthrough, which could be mitigated through loosening refinement criterion. However, doing so will counteract the computational efficiency improvement. With the promising results from two-phase flow problems, we plan to test our algorithm on more complex models such as 3-D three-phase black oil system in the near future.

\section*{Acknowledgement}
\par The author would like to thank the Center for Subsurface Modeling for supporting this research. We would also like to extend our gratitude to Gurpreet Singh and Uwe Köcher for the discussion related to this work.

\appendix

\section{Fully Discrete Formulation}
\label{sec:ful}

\par Consider the oil-water system, the variational form of Eqn.\eqref{eq:mb} through \eqref{eq:ic} is: find $\bs{u}_{\alpha,h}^{n}\in \bs{V}_{h}^{n,*}$, $\tilde{\bs{u}}_{\alpha,h}^{n}\in \bs{V}_{h}^{n,*}$, $s_{w,h}^{n}\in W_{h}^{n}$, $p_{o,h}^{n}\in W_{h}^{n}$ such that
\begin{equation}
  \bigg(\frac{\partial}{\partial t}\phi\Big(\rho_{w}s_{w,h}^{n}+\rho_{o}(1-s_{w,h}^{n})\Big),w\bigg)+\bigg(\nabla\cdot\Big(\bs{u}_{w,h}^{n}+\bs{u}_{o,h}^{n}\Big),w\bigg)=\bigg(q_{w}+q_{o},w\bigg)
  \label{eq:tot}
\end{equation}
\begin{equation}
  \bigg(\frac{\partial}{\partial t}\Big(\phi\rho_{w}s_{w,h}^{n}\Big),w\bigg)+\bigg(\nabla\cdot\bs{u}_{w,h}^{n},w\bigg)=\bigg(q_{w},w\bigg)
  \label{eq:wat}
\end{equation}
\begin{equation}
  \Big(K^{-1}\tilde{\bs{u}}_{o,h}^{n},\bs{v}\Big)-\Big(p_{o,h}^{n},\nabla\cdot\bs{v}\Big)=0
  \label{eq:omfe}
\end{equation}
\begin{equation}
  \Big(K^{-1}\tilde{\bs{u}}_{w,h}^{n},\bs{v}\Big)-\Big(p_{w,h}^{n},\nabla\cdot\bs{v}\Big)=-\Big(p_{c},\nabla\cdot\bs{v}\Big)
  \label{eq:wmfe}
\end{equation}
for all $\bs{v}\in\bs{V}_{h}^{n,*}$ and $w\in W_{h}^{n}$. The conversion between auxiliary and actual phase flux is referred to Eq.\eqref{eq:aux}. The oil saturation and water pressure are eliminated by the saturation constrain and the capillary pressure relation (assume oil phase being the non-wetting phase). 
\par For the fully discrete formulation, we will start by stating the basis functions in $RT_{0}\times DG_{0}$ discretization scheme. In spatial dimensions, the pressure and saturation are piecewise constants while velocity is piecewise linear. Meanwhile all variables are piecewise constants in temporal dimension as stated in Section \ref{sec:math}. Let $E_{i}^{m}=(\tau_{m},\tau_{m+1}] \times \Omega_{i}$ be a space-time element, we have
\begin{equation}
  w_{i}^{m}=
  \begin{cases}
  \begin{aligned}
    1 &\quad on\ E_{i}^{m}=\tau_{m}<t\leq \tau_{m+1} \bigcap x_{i-\frac{1}{2}}\leq x\leq x_{i+\frac{1}{2}} \\
    0 &\quad otherwise
  \end{aligned}
  \end{cases}
  \label{eq:cbase}
\end{equation}
\begin{equation}
  \bs{\varphi}_{i+\frac{1}{2}}^{m}=
  \begin{cases}
  \begin{aligned}
    \frac{x-x_{i-\frac{1}{2}}}{\big|E_{i}^{m}\big|}      &\quad on\ E_{i}^{m} \\
    \frac{x_{i+\frac{3}{2}}-x}{\big|E_{i+1}^{m}\big|} &\quad on\ E_{i+1}^{m}
  \end{aligned}
  \end{cases}
  \label{eq:lbase}
\end{equation}
The solution to Eqn.\eqref{eq:tot} through \eqref{eq:aux} can be written in discrete form using basis functions as
\begin{equation}
  \begin{cases}
  \begin{aligned}
    &p_{o}=\displaystyle\sum_{m=1}^{q}\sum_{i=1}^{r}P_{i}^{m}w_{i}^{m} \\
    &s_{w}=\displaystyle\sum_{m=1}^{q}\sum_{i=1}^{r}S_{w,i}^{m}w_{i}^{m} \\
    &\bs{u}_{\alpha}=\displaystyle\sum_{m=1}^{q}\sum_{i=1}^{r+1}U_{\alpha,i+\frac{1}{2}}^{m}\bs{\varphi}_{i+\frac{1}{2}}^{m} \\
    &\tilde{\bs{u}}_{\alpha}=\displaystyle\sum_{m=1}^{q}\sum_{i=1}^{r+1}\tilde{U}_{\alpha,i+\frac{1}{2}}^{m}\bs{\varphi}_{i+\frac{1}{2}}^{m} \\
  \end{aligned}
  \end{cases}
  \label{eq:disol}
\end{equation}
We remove the superscript $n$ and subscript $h$ in the above solution variables for this section since we need to use $n$ to pair basis functions. While keeping the solution in discrete form, we now substitute the testing functions in the variational forms of mass conservation and constitutive equation with $w_{j}^{n}$ and $\bs{\varphi}_{j+\frac{1}{2}}^{n}$. For the first term in Eqn.\eqref{eq:omfe} and \eqref{eq:wmfe} we obtain
\begin{equation}
  \Big(K^{-1}\tilde{\bs{u}}_{\alpha},\bs{\varphi}_{j+\frac{1}{2}}^{n}\Big)_{\Omega\times J}=\Bigg(K^{-1}\displaystyle\sum_{m=1}^{q}\sum_{i=1}^{r+1}\tilde{U}_{\alpha,i+\frac{1}{2}}^{m}\bs{\varphi}_{i+\frac{1}{2}}^{m},\bs{\varphi}_{j+\frac{1}{2}}^{n}\Bigg)_{\Omega\times J}=\frac{1}{2\Big|e_{j+\frac{1}{2}}^n\Big|}\Bigg(\frac{x_{j+\frac{1}{2}}-x_{j-\frac{1}{2}}}{K_{j}}+\frac{x_{j+\frac{3}{2}}-x_{j+\frac{1}{2}}}{K_{j+1}}\Bigg)\tilde{U}_{\alpha,j+\frac{1}{2}}^{n}
\label{eq:fir}
\end{equation}
Here, $\Big|e_{j+\frac{1}{2}}^n\Big|$ is an edge of a space-time element. Since the framework uses backward Euler scheme in time to avoid Courant-Fredricks-Levy condition, we have the construction
\begin{equation}
  \bs{\varphi}_{i+\frac{1}{2}}^{m}(e_{j+\frac{1}{2}}^{n})=
  \begin{cases}
  \begin{aligned}
    &\frac{1}{\Big|e_{j+\frac{1}{2}}^n\Big|} &\quad as\ i=j\ and\ m=n \\
    &0                                                       &\quad otherwise
  \end{aligned}
  \end{cases}
  \label{eq:stcon}
\end{equation}
The second term in Eqn.\eqref{eq:omfe} and \eqref{eq:wmfe} can be reformulated as
\begin{equation}
  \Big(p_{\alpha},\nabla\cdot\bs{\varphi}_{j+\frac{1}{2}}^{n}\Big)_{\Omega\times J}=\Bigg(\displaystyle\sum_{m=1}^{q}\sum_{i=1}^{r}P_{\alpha,i}^{m}w_{i}^{m},\nabla\cdot\bs{\varphi}_{j+\frac{1}{2}}^{n}\Bigg)_{\Omega\times J}=\int\limits_{E_{j}^{n}}\frac{P_{\alpha,j}^{n}}{\big|E_{j}^{n}\big|}-\int\limits_{E_{j+1}^{n}}\frac{P_{\alpha,j+1}^{n}}{\big|E_{j+1}^{n}\big|}=P_{\alpha,j}^{n}-P_{\alpha,j+1}^{n}
  \label{eq:sec}
\end{equation}
When non-matching grid caused by different time scales at $(j+\frac{1}{2})^-$ and $(j+\frac{1}{2})^+$ is encountered, assume the ratio between coarse and fine time step to be $\frac{\delta t_{c}}{\delta t_{f}}=\ell$, then for each $0 \leq k \leq \ell-1$
\begin{equation}
    \Big(p_{\alpha},\nabla\cdot\bs{\varphi}_{j+\frac{1}{2}}^{n-\frac{1}{\ell}k}\Big)_{\Omega\times J}=\Bigg(\displaystyle\sum_{m=1}^{q}\sum_{i=1}^{r}P_{\alpha,i}^{m}w_{i}^{m},\nabla\cdot\bs{\varphi}_{j+\frac{1}{2}}^{n-\frac{1}{\ell}k}\Bigg)_{\Omega\times J}=P_{\alpha,j}^{n-\frac{1}{\ell}k}-P_{\alpha,j+1}^{n}
  \label{eq:nm}
\end{equation}
The variational form of capillary pressure term can be revised in similar way as Eqn.\eqref{eq:sec} and \eqref{eq:nm}. Now we evaluate the mass conservation equation. The first term in Eqn.\eqref{eq:wat} becomes
\begin{equation}
  \bigg(\frac{\partial}{\partial t}\displaystyle\sum_{m=1}^{q}\sum_{i=1}^{r}\phi\rho_{w}s_{w,i}^{m}w_{i}^{m},w_{j}^{n}\bigg)_{\Omega\times J}=\Big((\phi\rho_{w}S_{w})_{j}^{n}-(\phi\rho_{w}S_{w})_{j}^{n-1}\Big)\big|E_{j}^{n-1}\big|
  \label{eq:wfir}
\end{equation}
In fine time scales, Eqn.\eqref{eq:wfir} can be altered as follow.
\begin{equation}
  \bigg(\frac{\partial}{\partial t}\displaystyle\sum_{m=1}^{q}\sum_{i=1}^{r}\phi\rho_{w}s_{w,i}^{m}w_{i}^{m},w_{j}^{n-\frac{1}{\ell}k}\bigg)_{\Omega\times J}=\Big((\phi\rho_{w}S_{w})_{j}^{n-\frac{1}{\ell}k}-(\phi\rho_{w}S_{w})_{j}^{n-\frac{1}{\ell}(k+1)}\Big)\Big|E_{j}^{n-\frac{1}{\ell}(k+1)}\Big|
  \label{eq:wfin}
\end{equation}
The second term is calculated as
\begin{equation}
    (\nabla\cdot\bs{u}_{w},w_{j}^{n})_{\Omega\times J}=\bigg(\nabla\cdot\displaystyle\sum_{m=1}^{q}\sum_{i=1}^{r+1}U_{\alpha,i+\frac{1}{2}}^{m}\bs{\varphi}_{i+\frac{1}{2}}^{m},w_{j}^{n}\bigg)_{\Omega\times J}=U_{w,j+\frac{1}{2}}^{n}-U_{w,j-\frac{1}{2}}^{n}
   \label{eq:wsec}
\end{equation}
The approach to handle non-matching grid is a little different for this term. Assume the fine time partition stays on $(j+\frac{1}{2})^-$ side, then on fine time elements we have
\begin{equation}
  (\nabla\cdot\bs{u}_{w},w_{j}^{n-\frac{1}{\ell}k})=U_{w,j+\frac{1}{2}}^{n-\frac{1}{\ell}k}-U_{w,j-\frac{1}{2}}^{n-\frac{1}{\ell}k}
  \label{eq:wsecf}
\end{equation}
while for the coarse time element we have
\begin{equation}
  (\nabla\cdot\bs{u}_{w},w_{j+1}^{n})=U_{w,j+\frac{3}{2}}^{n}-\displaystyle\sum_{k=0}^{\ell-1}U_{w,j+\frac{1}{2}}^{n-\frac{1}{\ell}k}
  \label{eq:wsecc}
\end{equation}
Eqn.\eqref{eq:wfin} and \eqref{eq:wsecc} will cause the accumulation and transmissibility matrix to have extra temporal bands forming in the lower triangle, making the corresponding linear system non-symmetric. The oil phase mass conservation equation is similar. Combining the equations for both phases will provide the expression for the total mass conservation equation. The two sides of Eqn.\eqref{eq:aux} is estimated as
\begin{equation}
  \begin{split}
  \begin{aligned}
    (\bs{u}_{\alpha},\bs{v})&=\displaystyle\sum_{m=1}^{q}\sum_{i=1}^{r+1}U_{\alpha,i+\frac{1}{2}}^{m}\Big(\bs{\varphi}_{i+\frac{1}{2}}^{m},\bs{\varphi}_{j+\frac{1}{2}}^{n}\Big)=\frac{x_{j+\frac{3}{2}}-x_{j-\frac{1}{2}}}{2\big|e_{j+\frac{1}{2}}^{n}\big|}U_{\alpha,j+\frac{1}{2}}^{n}
  \end{aligned}
  \end{split}
  \label{eq:auxl}
\end{equation}
\begin{equation}
  (\lambda_{\alpha}\tilde{\bs{u}}_{\alpha},\bs{v})\approx(\lambda_{\alpha}^{*}\tilde{\bs{u}}_{\alpha},\bs{v})=\frac{x_{j+\frac{3}{2}}-x_{j-\frac{1}{2}}}{2\big|e_{j+\frac{1}{2}}^{n}\big|}\lambda_{\alpha,j+\frac{1}{2}}^{*,n}\tilde{U}_{\alpha,j+\frac{1}{2}}^{n}
  \label{eq:auxr}
\end{equation}
The $\lambda_{\alpha,j+\frac{1}{2}}^{*,n}$ is the upwind mobility for stable numerical solution and is defined as
\begin{equation}
  \lambda_{\alpha,j+\frac{1}{2}}^{*,n}=\rho_{\alpha,j+\frac{1}{2}}^{n}\frac{k_{r\alpha,j+\frac{1}{2}}^{*}}{\mu_{\alpha}}=
  \begin{cases}
  \begin{aligned}
    &\frac{1}{2\mu_{\alpha}}(\rho_{\alpha,j}^{n}+\rho_{\alpha,j+1}^{n})k_{r\alpha}(S_{\alpha,j}^{n})     &\quad if\ \tilde{U}_{\alpha,j+\frac{1}{2}}^{n}>0 \\
    &\frac{1}{2\mu_{\alpha}}(\rho_{\alpha,j}^{n}+\rho_{\alpha,j+1}^{n})k_{r\alpha}(S_{\alpha,j+1}^{n}) &\quad otherwise
  \end{aligned}
  \end{cases}
  \label{eq:uw}
\end{equation}
The matrix corresponding to the above discrete formulation has sparsity pattern of three, five or seven non-zero diagonals, depending on the spatial dimension of the problem, with one extra temporal diagonal in the lower triangle. Forming such matrix in block format is referred to \cite{Singh:0818, Ganis:1014, Singh:1218}. 

\bibliographystyle{plain} 
{\small \bibliography{draft.bbl}}

\end{document}